\documentclass{article}

\usepackage{amsfonts}
\usepackage{amssymb}
\usepackage{amsmath}
\usepackage{amsthm}
\usepackage{amsopn}
\usepackage{graphicx}
\graphicspath{{./pictures/}}
\usepackage{subfigure}
\usepackage{bm}
\usepackage{fancyhdr}
\usepackage[a4paper]{geometry}

\usepackage[colorlinks=true,%
citecolor=black,%
filecolor=black,%
linkcolor=black,%
urlcolor=blue,%
pdfauthor={Joscha Gedicke and Arbaz Khan},%
pdftitle={Arnold-Winther Mixed Finite Elements for Stokes Eigenvalue Problems}]{hyperref}

\newtheorem{theorem}{Theorem}

\newtheorem{lemma}[theorem]{Lemma}

\theoremstyle{definition}

\newtheorem{definition}[theorem]{Definition}

\numberwithin{theorem}{section}

\newcommand{\TheTitle}{Arnold-Winther Mixed Finite Elements for Stokes Eigenvalue Problems} 
\newcommand{\TheShortTitle}{Arnold-Winther MFEM for Stokes Eigenvalue Problems} 
\newcommand{\TheAuthors}{J. Gedicke, A. Khan}

\DeclareMathOperator{\bmdiv}{\mathbf{div}}
\DeclareMathOperator{\ddiv}{div}

\newcommand{\figurewidtha}{0.48}
\newcommand{\figurewidthb}{0.48}

\title{{\TheTitle}%
\thanks{
The first author has been funded by the Austrian Science Fund (FWF) 
through the project P 29197-N32.
The second author has been funded
by the Mathematics Center Heidelberg (Match) at University of Heidelberg, Germany.}}

\author{
  Joscha Gedicke\thanks{ Faculty of Mathematics, University of Vienna, 1090 Vienna, Austria
    (\mbox{joscha.gedicke@univie.ac.at}).}
  \and
  Arbaz Khan\thanks{Corresponding author: 
 School of Mathematics,  University of Manchester, 
M13 9PL Manchester, UK (\mbox{arbaz.khan@manchester.ac.uk}).}
}

\date{}

\begin{document}

\maketitle

\begin{abstract}
This paper is devoted to study the Arnold-Winther mixed finite element method 
for two dimensional Stokes eigenvalue problems using the stress-velocity formulation. 
A~priori error estimates for the eigenvalue and eigenfunction errors are presented. 
To improve the approximation for both eigenvalues and eigenfunctions, we propose a local post-processing. 
With the help of the local post-processing, we derive a reliable a~posteriori error estimator which is shown
to be empirically efficient.
We confirm numerically the proven higher order convergence of the post-processed eigenvalues for convex domains
with smooth eigenfunctions.
On adaptively refined meshes we obtain numerically optimal higher orders of convergence of the post-processed eigenvalues 
even on nonconvex domains.
\end{abstract}

{\small\noindent\textbf{Keywords}
a priori analysis, a posteriori analysis, Arnold-Winther finite element, mixed finite element, Stokes eigenvalue problem

\noindent
\textbf{AMS subject classification}
65N15,   	
65N25,   	
65N30   	
}

\section{Introduction}
\label{intro}
Over the last decade, the numerical analysis of the finite element method for eigenvalue problems
has been of increasing interest because of various practical applications.
Specifically, the numerical analysis of the Stokes eigenvalue problem is a broad research area. 
Huang et al. \cite{PHYHXF} discuss the numerical analysis
of several stabilized finite element methods for the Stokes eigenvalue problem. In \cite{MSDMRR}, Meddahi et al. proposed a finite 
element analysis of a pseudo-stress formulation for the Stokes eigenvalue problem. In \cite{OTDBRC}, T\"{u}rk et al. introduced a 
stabilized finite element method for two-field and three-field Stokes eigenvalue problems.
From \cite{BJJP}, one can also study a variety of mixed or hybrid finite element methods
for eigenvalue problems.

In the literature, most of the results based on the a~posteriori error analysis for the finite element method 
(see \cite{MAJTO,RV} and the references therein) consider the source problem. 
In comparison there are only a few results for the a~posteriori error analysis of the Stokes eigenvalue problem available.
In \cite{LCMLRS}, Lovadina et al. presented the numerical analysis for a residual-based a~posteriori error estimator for the 
finite element discretization of the Stokes eigenvalue problem. 
In \cite{LHGWWSYN}, Liu et al.  proposed the finite element approximation of the Stokes eigenvalue 
problem based on a projection method, and derive some superconvergence results and the related recovery type a~posteriori error 
estimators. A posteriori error estimators for stabilized low-order mixed finite elements for the Stokes eigenvalue problem are presented by 
Armentano et al. \cite{AMGVM}. A new adaptive mixed finite element method based on residual type a~posteriori error estimators for the 
Stokes eigenvalue problem is proposed by Han et al. \cite{HJZZYY}. In \cite{PH}, Huang presented two stabilized finite element 
methods for the Stokes eigenvalue problem based on the lowest equal-order finite element pair and also discussed 
a~posteriori lower and upper bounds of Stokes eigenvalues.

Arnold and Winther introduced the strongly symmetric Arnold-Winther mixed finite element (MFEM) for linear elastic problems in 
\cite{ADRW} and proved its stability for any material parameters.  Hence, the proposed Arnold-Winther MFEM is also stable for the 
Stokes problem as a limit case of linear elasticity. In \cite{CCJGEJP}, Carstensen et al. presented the Arnold-Winther mixed finite element 
formulation for the Stokes source problem. 

In this paper, we are presenting the Arnold-Winther mixed finite element formulation for the two dimensional Stokes eigenvalue problem 
using the stress-velocity formulation. In principal, the  stress-velocity formulation is originated from a physical model where 
incompressible Newtonian flows are modeled by the conservation of momentum and the constitute law. 
The stress-velocity formulation for the Stokes eigenvalue problem reads as follows:
find a symmetric stress tensor $\bm{\sigma}$, a nonzero eigenfunction $\bm{u}$, and an eigenvalue $\lambda$ such that
\begin{align*}
    -\bmdiv  \bm{\sigma}& =\lambda \bm{u}\quad \mbox{in}\;\Omega,\quad
    \mathcal{A}\bm{\sigma}-\bm{\epsilon}(\bm{u})=0\quad\mbox{in}\;\Omega,\quad
    \bm{u}=0\quad \mbox{on}\;\partial\Omega,
\end{align*}
for a bounded Lipschitz domain $\Omega\subset \mathbb{R}^2$. 
Here $\bm{\sigma}, \mathcal{A}\bm{\sigma} $ and $\bm{\epsilon}(\bm{u})$ 
are the stress tensor,  the deviatoric stress tensor and the deformation rate tensor, respectively. 

For simplicity, we restrict ourselves in this paper to the analysis for simple eigenvalues $\lambda$.
This simplifies the a priori error analysis in that the eigenspace to $\lambda$ is one-dimensional.
Note that the presented a posteriori error estimator can directly be extended to multiple eigenvalues,
in the way that the sum over the squared error estimators, for all the discrete eigenvectors that approximate eigenvectors
to the same multiple eigenvalue, provides an upper bound of the eigenvalue error.

The remaining parts of this paper are organized as follows: 
Section~\ref{formu} presents the necessary notation, the formulation of the problem, 
the discretization of the domain and the mixed finite element formulation. 
Section~\ref{numep} is devoted to study the a~priori error analysis of the eigenvalue problem.
Section~\ref{secpost} presents the  local post-processing and its higher order convergence analysis. 
The a~posteriori error analysis is introduced in Section~\ref{postera}. 
Finally, we verify the reliability and efficiency
of the a~posteriori error estimator 
and the higher order convergence of post-processed eigenvalues
in three numerical experiments
in Section~\ref{comre}.

\section{Preliminaries} \label{formu}
Let $H^s(\omega)$ be the standard Sobolev space  with the associated norm $\|\cdot\|_{s,\omega}$ for $s\ge0$. In case of $
\omega=\Omega$, we use $\|\cdot\|_{s}$ instead of $\|\cdot\|_{s,\Omega}$. Let $H^{-s}(\omega) :=(H^{s}(\omega))^*$ be the dual 
space of $H^{s}(\omega)$. Now we extend the definitions for vector and matrix-valued function. Let $\bm{H}^{s}(\omega)=\bm{H}^{s}
(\omega;\mathbb{R}^2)$ and  $\bm{H}^{s}(\omega,\mathbb{R}^{2\times 2})$ be the Sobolev spaces over the set of $2$-dimensional 
vector and $2\times 2$ matrix-valued function, respectively.
Define $\bm{v}=(v_1,v_2)^t\in\mathbb{R}^2, \bm{\tau}=(\tau_{ij})_{2\times 2}$ and $\bm{\sigma}=(\sigma_{ij})_{2\times 2}\in\mathbb{R}
^{2\times2}$, then
\begin{align*}
    \nabla\bm{v}&:=\Bigg(\begin{array}{cc}
    \frac{\partial v_1}{\partial x}&\frac{\partial v_1}{\partial y}\\
    \frac{\partial v_2}{\partial x}&\frac{\partial v_2}{\partial y}
    \end{array}\Bigg),\quad \bmdiv(\tau):=\Bigg(\begin{array}{c}
    \frac{\partial \tau_{11}}{\partial x}+\frac{\partial \tau_{12}}{\partial y}\\
    \frac{\partial \tau_{21}}{\partial x}+\frac{\partial \tau_{22}}{\partial y}
    \end{array}\Bigg),\quad\mbox{tr}\,\bm{\tau} :=\tau_{11}+\tau_{22},\\
    \bm{\tau\,v} &:=\Big(\begin{array}{c}
    \tau_{11}v_1+\tau_{12}v_2\\
    \tau_{21}v_1+\tau_{22}v_{2}
    \end{array}\Big),\quad\bm{\tau} :\bm{\sigma}:=\sum_{i,j}\tau_{ij}\sigma_{ij},\quad\bm{\delta}:=2\times2\; \mbox{unit matrix}. 
\end{align*}
Let the  divergence conforming stress space $\bm{H}(\bmdiv,\Omega,\mathbb{R}^{2\times 2})$ be defined as
\begin{align*}
    \bm{H}(\bmdiv,\Omega,\mathbb{R}^{2\times 2})
    :=\{\bm{\tau}\in \bm{L}^2(\Omega;\mathbb{R}^{2\times2}) | \bmdiv\bm{\tau}\in \bm{L}^2(\Omega)\},
\end{align*}
equipped with the norm
\begin{align*}
    \|\bm{\tau}\|_{\bm{H}(\bmdiv,\Omega,\mathbb{R}^{2\times 2})}^2 :=(\bm{\tau},\bm{\tau})+(\bmdiv\bm{\tau},\bmdiv\bm{\tau}).
\end{align*}
In this paper, we are using the notation $(\cdot,\cdot)_{\omega}$ for the $\bm{L}^2(\omega;\mathbb{R}^{2\times2})$ inner product 
$\int_\omega \bm{\tau}:\bm{\tau}dx$ as well as the $\bm{L}^2(\omega)$ inner product $\int_\omega \bm{\tau}\cdot\bm{\tau}dx$. 
For $w=\Omega$, the notation $(\cdot,\cdot)$ is used instead of $(\cdot,\cdot)_{\Omega}$.
The symbols $\lesssim$ and $\gtrsim$ are used throughout the paper to denote inequalities 
which are valid up to positive constants that are
independent of the local mesh size $h$ but may depend on the size of the
eigenvalue $\lambda$ and the coefficient $\nu$.
\par
Let $\Omega\subset \mathbb{R}^2$ be a bounded and connected Lipschitz domain. 
Consider the Stokes eigenvalue problem
\begin{align}\label{pre12}
\begin{split}
    -\nu \bigtriangleup \bm{u}+\nabla p &=\lambda \bm{u}\quad \,\mbox{in }\Omega,  \\
    \ddiv\bm{u}&=0\quad\quad\mbox{in }\Omega,\\
    \bm{u}&=0\quad\quad\mbox{on }\partial\Omega,
\end{split}
\end{align}  
with the compatibility condition on the pressure
\begin{align}\label{pre11}
    \int_{\Omega} p\;dx=0.
\end{align}
Define $\bm{\sigma}=(\sigma_{i,j})_{2\times 2}$ as a stress tensor and the deformation rate tensor as
\begin{align*}
    \bm{\epsilon}(\bm{u}):=\frac{1}{2}(\nabla \bm{u}+(\nabla \bm{u})^t).
\end{align*} 
From \eqref{pre12} we can derive 
the stress-velocity-pressure formulation for the Stokes eigenvalue problem, 
which is the set of original physical equations for 
incompressible Newtonian flow,
\begin{align*}
    -\bmdiv  \bm{\sigma}& =\lambda \bm{u}\quad \,\mbox{in }\Omega,\\
    \bm{\sigma}+p\bm{\delta}-2\nu\bm{\epsilon}(\bm{u})&=0\quad\quad \mbox{in }\Omega,\\
    \ddiv\bm{u}&=0\quad\quad \mbox{in }\Omega,\\
    \bm{u}&=0\quad\quad\mbox{on }\partial\Omega.
\end{align*}
Next, we define the deviatoric operator
$\mathcal{A}:\mathbb{S}\rightarrow \mathbb{S}$,
where  $\mathbb{S}$ is the space of symmetric tensors 
$\mathbb{S}=\{\bm{\tau}\in \mathbb{R}^{2\times 2}|\bm{\tau}=\bm{\tau}^t\}$. 
Then we can define the deviator $\mathcal{A}\bm{\tau}$ of $\bm{\tau}$ by
\begin{align*}
    \mathcal{A}\bm{\tau} :=\frac{1}{2}\big(\bm{\tau}-\frac{1}{2}(\mbox{tr}\; \bm{\tau})\bm{\delta}\big)
    \quad \mbox{for all}\; \bm{\tau}\in \mathbb{S}. 
\end{align*} 
Here, $\mbox{Ker}(\mathcal{A})=\{q\bm{\delta}\in\mathbb{S}\;|\;q\in \mathbb{R}\}$ 
and $\mathcal{A}\tau$ is a trace-free tensor.
In addition, $\mathcal{A}$ satisfies for all $\bm{\tau},\bm{\sigma}\in \mathbb{S}$
the following properties, 
\begin{align*}
    (\mathcal{A}\bm{\tau},\bm{\sigma}) &=(\bm{\tau},\mathcal{A}\bm{\sigma}),\\
    (\mathcal{A}\bm{\tau},2\nu \mathcal{A}\bm{\sigma}) &=(\mathcal{A}\bm{\sigma},\bm{\tau})
    =\frac{1}{2\nu}\Big((\bm{\sigma},\bm{\tau})-\frac{1}{2}(\mbox{tr}\;\bm{\sigma},\mbox{tr}\;\bm{\tau})\Big),\\
    \|\mathcal{A}\bm{\tau}\|&\le\frac{1}{2}\|\bm{\tau}\|.
\end{align*}
\par
Using the deviatoric tensor $\mathcal{A}$, we arrive at the stress-velocity formulation of the Stokes eigenvalue problem
\begin{align}\label{prob11}
\begin{split}
    -\bmdiv  \bm{\sigma}& =\lambda \bm{u}\quad \,\mbox{in }\Omega,\\
    \mathcal{A}\bm{\sigma}-\bm{\epsilon}(\bm{u})&=0\quad\quad \mbox{in }\Omega,\\
    \bm{u}&=0\quad\quad\mbox{on }\partial\Omega.
\end{split}
\end{align}
Using the compatibility condition \eqref{pre11}, we have
\begin{align*}
    \int_{\Omega}\mbox{tr} \;\bm{\sigma}\,dx=0.
\end{align*}
Defining $\bm{V} :=\bm{L}^2(\Omega)$ and
\begin{align*}
    \bm{\Phi} :=\bm{H}(\bmdiv,\Omega;\mathbb{S})/\mathbb{R}
    \sim \Big\{\bm{\tau}\in\bm{H}(\bmdiv,\Omega,\mathbb{S})|\int_{\Omega}\mbox{tr}\,\bm{\tau}\, dx=0\Big\},
\end{align*}
we can now derive the following weak form for problem \eqref{prob11}:
find $\bm{\sigma}\in\bm{\Phi}$,
$\bm{u}\in\bm{V}$ with $\|\bm{u}\|_0=1$,
and $\lambda\in \mathbb{R}_+$ such that
\begin{align}\label{prob12}
\begin{split}
    (\mathcal{A}\bm{\sigma},\bm{\tau})+(\bmdiv  \bm{\tau},\bm{u})
    &=0\qquad\qquad\;\mbox{for all}\;\bm{\tau}\in{\bm{\Phi}},\\
    (\bmdiv  \bm{\sigma}, \bm{v})& =-\lambda (\bm{u},\bm{v})\quad\mbox{for all}\;\bm{v}\in{\bm{V}}.
\end{split}
\end{align}
\par
Let $\{\mathcal{T}_{h}\}$ denote a family of regular triangulations of $\bar{\Omega}$ into triangles $K$ of diameter $h_K$. For 
each $\mathcal{T}_h$, we define $\mathcal{E}_h$ as the set of all edges of $\mathcal{T}_h$ and
$h_E$ as the length of the edge $E\in\mathcal{E}_h$. Furthermore, let $[\bm{w}]$ denote the jump of $\bm{w}$,
\begin{align*}
    [\bm{w}]|_{E}:=(\bm{w}|_{K_+})|_{E}-(\bm{w}|_{K_-})|_{E}\quad\mbox{if}\quad E=\bar{K}_{+}\cap\bar{K}_{-}.
\end{align*}
Let $\nabla_{\mathcal{T}_h}$ and $\bm{\epsilon}_{\mathcal{T}_h}$ denote the piecewise gradient and
the piecewise symmetric gradient, i.e.
$(\nabla_{\mathcal{T}_h}\cdot)|_{K}=\nabla(\cdot|_{K})$ and $(\bm{\epsilon}_{\mathcal{T}_h}\cdot)|_{K}=\bm{\epsilon}(\cdot|_{K})$,
for all $K\in\mathcal{T}_h$.
\par
Finally, we define the following finite element spaces associated 
with the triangulation $\mathcal{T}_h$,
\begin{align*}
    AW_{k}(K)&:=\{\bm{\tau}\in P_{k+2}(K;\mathbb{S})\;|\;\bmdiv\bm{\tau}\in P_{k}(K;\mathbb{R}^2)\},\\
    \bm{\Phi}_h&:=\{\bm{\tau}\in \bm{\Phi}\;|\;\bm{\tau}|_K\in AW_{k}(K)\},\\
    \bm{{V}}_h&:=\{\bm{v}\in \bm{L}^{2}(\Omega)\;|\;\bm{v}|_K\in P_{k}(K;\mathbb{R}^2)\}.
\end{align*}
Here, $AW_{k}(K)$ denotes the Arnold-Winther MFEM of index $k\ge 1$ of \cite{ADRW} and $P_{k}(K)$ denotes the set of all 
polynomials of total degree up to $k$ on the domain $K$. The space $\Phi_h$ consists of all symmetric polynomial matrix fields of 
degree at most $k+1$ together with the divergence free matrix fields of degree $k+2$.
Note that $\bm{\Phi}_{h}\subset \bm{\Phi}$. When $\bm{\tau}_h\in \bm{\Phi}_h$ we have that $\bm{\tau}_{h}$ has continuous normal 
components and $\bm{\tau}_{h}$ satisfies $\int_{\Omega}\mbox{tr}\, \bm{\tau}_h dx=0$.
\par
The Arnold-Winther MFEM of the Stokes eigenvalue problem becomes the following: find $\bm{\sigma}_h\in\bm{\Phi}_h$,  
$\bm{u}_h\in\bm{V}_h$ with $\|\bm{u_h}\|_0=1$, and $\lambda_h\in\mathbb{R}_+$ such that
\begin{align}\label{mfpro11}
\begin{split}
    (\mathcal{A}\bm{\sigma}_h,\bm{\tau}_h)+(\bmdiv  \bm{\tau}_h,\bm{u}_h)&=0
    \qquad\qquad\qquad\!\mbox{for all}\;\bm{\tau}_h\in{\bm{\Phi}_h},\\
    (\bmdiv  \bm{\sigma}_h, \bm{v}_h)& =-\lambda_h (\bm{u}_h,\bm{v}_h)\quad\mbox{for all}\;\bm{v}_h\in{\bm{V}_h}.
\end{split}
\end{align}
%
\section{A priori error analysis}\label{numep}
Our main aim is to show that the solutions of the Arnold-Winther MFEM of the Stokes 
eigenvalue problem
converge to the solution of the corresponding spectral problem which comes to apply the classical spectral approximation theory 
presented in \cite{IBJO} to the associated source problem.
\par
Consider the stress-velocity formulation for the associated source problem
\begin{align}\label{sprob11}
\begin{split}
    -\bmdiv  \bm{\sigma^f}& =\bm{f}\quad \mbox{in}\;\Omega,\\
    \mathcal{A}\bm{\sigma^f}-\bm{\epsilon}(\bm{u^f})&=0\quad\; \mbox{in}\;\Omega,\\
    \bm{u^f}&=0\quad \;\mbox{on}\;\partial\Omega.
\end{split}
\end{align}
Using the compatibility condition \eqref{pre11}, we have
\begin{align*}
    \int_{\Omega}\mbox{tr} \;\bm{\sigma^f}\,dx=0.
\end{align*}
For the above problem \eqref{sprob11} there exist well-known regularity results 
for convex domains with sufficiently smooth boundary $\partial\Omega$.
If $\bm{f}\in \bm{L}^{2}(\Omega)$ then the solution of problem \eqref{sprob11} satisfies 
$\bm{u}\in \bm{H}^{2}(\Omega)\cap\bm{H}^{1}(\bar{\Omega})$, $p\in{H}^{1}(\Omega)/\mathbb{R}$, $\bm{\sigma}\in \bm{H}^{1}(\Omega;\mathbb{S})$, 
and 
\begin{align}\label{soreg11}
    \|\bm{u}\|_{2}+\|p\|_{1}+\|\bm{\sigma}\|_1\le\|\bm{f}\|_0 .
\end{align}
Using the well-posedness of problem \eqref{sprob11}, the 
operators $S:\bm{V}\rightarrow\bm{\Phi}$  and $T:\bm{V}\rightarrow\bm{V}$
are well defined for any $\bm{f}\in\bm{V}$ such that
$S\bm{f}=\bm{\sigma^f}$, and $T\bm{f}=\bm{u}^f$ are the stress 
and velocity solutions, respectively.
\par
We now define the following weak form for problem \eqref{prob11}:
for given $\bm{f}\in \bm{V}$, $(S\bm{f},T\bm{f})\in (\bm{\Phi},\bm{V})$ is the solution of
\begin{align}\label{nmsprob12}
\begin{split}
    (\mathcal{A}( S\bm{f}),\bm{\tau})+(\bmdiv  \bm{\tau},T\bm{f})&=0\qquad\qquad\mbox{for all }\bm{\tau}\in{\bm{\Phi}},\\
    (\bmdiv  ( S\bm{f}), \bm{v})& = -(\bm{f},\bm{v})\quad\,\mbox{for all }\bm{v}\in{\bm{V}}.
\end{split}
\end{align}
The above problem \eqref{nmsprob12} has unique solution from the well known 
inf-sup condition of the mixed formulation and \cite[Lemma~2.1]{CCJGEJP}.
\par
For the discrete solution operators $S_h:\bm{V}\rightarrow\bm{\Phi}_h$ and $T_h:\bm{V}\rightarrow\bm{V}_h$, 
the Arnold-Winther MFEM of the Stokes source problem becomes the following: 
find $S_h\bm{f}\in\bm{\Phi}_h$ and  $T_h\bm{f}\in\bm{V}_h$ such that
\begin{align}\label{nmsprob131}
\begin{split}
    (\mathcal{A}( S_h\bm{f}),\bm{\tau}_h)+(\bmdiv  \bm{\tau}_h,T_h\bm{f})
    &=0\quad\quad\quad\;\,\mbox{for all}\;\bm{\tau}_h\in{\bm{\Phi}_h},\\
    (\bmdiv  ( S_h\bm{f}), \bm{v}_h)& = -(\bm{f},\bm{v}_h)\quad\mbox{for all}\;\bm{v}_h\in{\bm{V}_h}.
\end{split}
\end{align}
From \cite{CCJGEJP}, the discrete source problem is well-posed and has a unique solution.
From \cite{ADRW} we have the following a~priori estimates 
for $(\bm{\sigma^f},\bm{u^f})\in(\bm{\Phi}\cap\bm{H}^{k+2}(\Omega;\mathbb{S}))\times\bm{H}^{k+2}(\Omega)$ 
and $\bm{f}\in \bm{H}^{k}(\Omega)$
\begin{align}\label{eqsoei11}
    \|S\bm{f}-S_h\bm{f}\|_0&\lesssim h^{m}\|\bm{\sigma}^f\|_{m},\quad 1\le m\le k+2,\\
    \|\bmdiv (S\bm{f}-S_h\bm{f})\|_0&\lesssim h^{m}\|\bmdiv \bm{\sigma}^f\|_{m},\quad 0\le m\le k+1,\label{eqsoei12}\\
    \|T\bm{f}-T_h\bm{f}\|_0&\lesssim h^{m}\|\bm{u}^f\|_{m+1},\quad 1\le m\le k+1.\label{eqsoei13}
\end{align}
\par
Hence, we can state the following convergence results by \eqref{eqsoei11} and \eqref{eqsoei13}
\begin{align}\label{convere11}
    \|T-T_h\|_{\mathcal{L}(\bm{V},\bm{V})}\rightarrow 0\quad\mbox{if}\quad h\rightarrow 0,\\
    \|S-S_h\|_{\mathcal{L}(\bm{V},\bm{\Phi})}\rightarrow 0\quad\mbox{if}\quad h\rightarrow 0.\label{convere12}
\end{align}
The above results \eqref{convere11} and  \eqref{convere12} are equivalent to the convergence of eigenvalues and eigenfunctions. 
Thus, using the abstract theory from \cite{DB,BJJP} 
and the a~priori results \eqref{eqsoei11} and \eqref{eqsoei13}, we have
\begin{align}
    \|\bm{u}-\bm{u}_h\|_{0}&\lesssim h^m, \quad 1\le m\le k+1,\label{erreigu11}\\
    \|\bm{\sigma}-\bm{\sigma}_h\|_{0}&\lesssim h^m,\quad 1\le m\le k+2.\label{erreigs12}
\end{align}
Note that $p=-tr\, \bm{\sigma}/2$, hence the approximation of the pressure is defined by $p_h=-tr\, \bm{\sigma}_h/2$ and satisfies the 
following estimate
\begin{align*}
    \|p-p_h\|_0=\frac{1}{2}\|\mbox{tr}\,\bm{\sigma}- \mbox{tr}\,\bm{\sigma}_h\|_0
    \leq\|\bm{\sigma}-\bm{\sigma}_h\|_0\lesssim h^{m},\; 1\le m\le k+2.
\end{align*}
The next lemma establishes a connection between the errors in the 
eigenvalues and in the eigenfunctions.
\begin{lemma}\label{eigerrlem11}
Let $(\bm{\sigma},\bm{u},\lambda)$ and $(\bm{\sigma}_h,\bm{u}_h,\lambda_h)$ 
be solutions of the continuous eigenvalue problem \eqref{prob12} and the discrete 
eigenvalue problem \eqref{mfpro11}, respectively. 
Then, we have the identity
\begin{align}\label{eigenlam11}
    \lambda-\lambda_h= 2\nu\|\mathcal{A}(\bm{\sigma}-\bm{\sigma}_h)\|^{2}_{0}-\lambda_h\|\bm{u}-\bm{u}_h\|^{2}_{0}.
\end{align}
\end{lemma}
\begin{proof}
From \eqref{prob12} and \eqref{mfpro11}, using $(\bm{u},\bm{u})=1$ and  $(\bm{u}_h,\bm{u}_h)=1$, we have
\begin{align*}
    (2\nu\mathcal{A}\bm{\sigma},\mathcal{A}\bm{\sigma})&=\lambda(\bm{u},\bm{u})=\lambda,\\
    (2\nu\mathcal{A}\bm{\sigma}_h,\mathcal{A}\bm{\sigma}_h)&=\lambda_h(\bm{u}_h,\bm{u}_h)=\lambda_h,\\
    (2\nu\mathcal{A}\bm{\sigma},\mathcal{A}\bm{\sigma}_h)&=\lambda_h(\bm{u}_h,\bm{u}).
\end{align*}
Moreover, it follows
\begin{align*}
    2\nu\|\mathcal{A}(\bm{\sigma}-\bm{\sigma}_h)\|_{0}^2
    &=2\nu(\mathcal{A}(\bm{\sigma}-\bm{\sigma}_h),\mathcal{A}(\bm{\sigma}-\bm{\sigma}_h))\\
    &=2\nu(\mathcal{A}\bm{\sigma},\mathcal{A}\bm{\sigma})+2\nu(\mathcal{A}\bm{\sigma}_h,\mathcal{A}\bm{\sigma}_h)
    -4\nu(\mathcal{A}\bm{\sigma},\mathcal{A}\bm{\sigma}_h)\\
    &=\lambda+\lambda_h-2\lambda_h(\bm{u}_h,\bm{u}).
\end{align*}
Then we obtain
\begin{align*}
    \lambda- \lambda_h=2\nu\|\mathcal{A}(\bm{\sigma}-\bm{\sigma}_h)\|_{0}^2-2\lambda_h+2\lambda_h(\bm{u}_h,\bm{u}).
\end{align*}
Using 
$\|\bm{u}-\bm{u}_h\|^2_{0}=2-2(\bm{u}_h,\bm{u})$ yields
\begin{align*}
    \lambda- \lambda_h=2\nu\|\mathcal{A}(\bm{\sigma}-\bm{\sigma}_h)\|_{0}^2-\lambda_h\|\bm{u}-\bm{u}_h\|^2_{0}.
\end{align*}
\end{proof}
An immediate consequence of Lemma~\ref{eigerrlem11}
are the following two lemmas.
\begin{lemma}
For sufficiently smooth
$\bm{u}\in \bm{H}^{k+2}(\Omega)$,
$\bm{\sigma}\in\bm{H}^{k+2}(\Omega;\mathbb{S})$,
the following a~priori error estimate holds
\begin{align}\label{eigenlam12}
   |\lambda - \lambda_h| \lesssim h^{2m}, \quad 0\le m\le k+1.
\end{align}
\end{lemma}
\begin{proof}
The assertion follows from \eqref{erreigu11}--\eqref{eigenlam11}.
\end{proof}
\begin{lemma}\label{eigerr2}
For sufficiently smooth
$\bm{u}\in \bm{H}^{k+2}(\Omega)$,
$\bm{\sigma}\in\bm{H}^{k+2}(\Omega;\mathbb{S})$,
the following a~priori error estimate holds
\begin{align}\label{erreigdiv}
    \|\bmdiv(\bm{\sigma}-\bm{\sigma}_h)\|_0\lesssim h^{m}, \quad 0\le m\le k+1.
\end{align}
\end{lemma}
\begin{proof}
From \eqref{prob12} and \eqref{mfpro11}, we have $\bmdiv(\bm{\sigma}) = -\lambda \bm{u}$ and $\bmdiv(\bm{\sigma}_h) = -\lambda_h \bm{u}_h$.
Therefore we get
\begin{align*}
\|\bmdiv(\bm{\sigma}-\bm{\sigma}_h) \|_0
&= \| \lambda_h \bm{u}_h - \lambda \bm{u}\|_0
\leq (\lambda_h-\lambda)\|\bm{u}_h\|_0 + \lambda\| \bm{u}_h - \bm{u}\|_0.
\end{align*}
The assertion follows from 
\eqref{erreigu11}--\eqref{eigenlam11}, and $\|\bm{u}_h\|_0=1$.
\end{proof}
\par
Let $\bm{P}_h$ denote the $\bm{L}^2$ projection onto $\bm{V}_h$ with the well known approximation property
\begin{align}\label{serr15}
    \|\bm{P}_h\bm{v}-\bm{v}\|_0&\lesssim h^{m}\|\bm{v}\|_{m},\quad\mbox{for all} \;\bm{v}\in\bm{H}^{m}(\Omega).
\end{align}
\par
In the following we relate the discrete eigenfunctions $(\bm{\sigma}_h,\bm{u}_h)$ 
to the discrete approximations $(\bm{\sigma}_{\lambda,h},\bm{u}_{\lambda,h})$
of the associated source problem \eqref{nmsprob131} with right hand side 
$\bm{f}=\lambda\bm{u}\in\bm{V}$.
Then  $\bm{\sigma}_{\lambda,h} = \lambda S_h\bm{u}$ and $\bm{u}_{\lambda,h} = \lambda T_h\bm{u}$
and we have the following lemma.
\begin{lemma}[{\cite[Theorem 3.4]{CCJGEJP}}]\label{mposterr12}
For sufficiently smooth boundary $\partial \Omega$, $\bm{\sigma}\in\bm{H}^{k+2}(\Omega;\mathbb{S})$, 
and $\bmdiv  \bm{\sigma}\in \bm{H}^{k+1}(\Omega)$, it holds that
\begin{align}
    \|\bm{P}_h\bm{u}-\bm{u}_{\lambda,h}\|_0&\lesssim h^{k+3}(\|\bm{\sigma}\|_{k+2}+\|\bmdiv\bm{\sigma}\|_{k+1}). 
\end{align}
\end{lemma}
Next, we prove a similar estimate for the eigenfunction $\bm{u}_h$ which is used 
in Section~\ref{secpost} to derive an error estimate for the post-processed eigenfunction.
\begin{theorem}\label{errthm11}
For sufficiently smooth boundary $\partial \Omega, \bm{\sigma}\in\bm{H}^{k+2}(\Omega;\mathbb{S})$, 
$\bmdiv\bm{\sigma}\in \bm{H}^{k+1}(\Omega)$, it holds that
\begin{align}
    \|\bm{P}_h\bm{u}-\bm{u}_h\|_0&\lesssim h^{k+3}( h^{k-1} + \|\bm{\sigma}\|_{k+2}+\|\bmdiv\bm{\sigma}\|_{k+1}).
\end{align}
\end{theorem}
The proof of Theorem~\ref{errthm11} does not follow directly from Lemma~\ref{mposterr12}, 
because the eigenvalue problem does not satisfy an orthogonality condition, i.e.
$(\bmdiv  (\bm{\sigma}-\bm{\sigma}_{h}), \bm{v}_h)\neq 0$,
therefore we need the following lemma.
\begin{lemma}\label{mposterr12311}
For sufficiently smooth boundary $\partial \Omega$, 
$\bm{\sigma}\in\bm{H}^{k+2}(\Omega;\mathbb{S})$, 
and $\bmdiv  \bm{\sigma}\in \bm{H}^{k+1}(\Omega)$,
the difference between the discrete eigenfunction $\bm{u}_h$ and the discrete solution of 
the associated source problem $\bm{u}_{\lambda,h}$ can be estimated by
\begin{align*}
    \|\bm{u}_{\lambda,h}-\bm{u}_h\|\lesssim h^{k+3}(h^{k-1} + \|\bm{\sigma}\|_{k+2}+\|\bmdiv \bm{\sigma}\|_{k+1}).
\end{align*}
\end{lemma}
\begin{proof}
First we define some relations
\begin{align*}
    \lambda T \bm{u}=\bm{u},\quad \lambda_h T_h \bm{u}_h
    =\bm{u}_h\quad \mbox{and}\quad \lambda T_h \bm{u}=\bm{u}_{\lambda,h}.
\end{align*}
Since $\lambda$ is a simple eigenvalue, the eigenspace of the eigenvalue $\lambda$ is spanned by $\bm{u}$. 
The operator $T:\bm{V}\rightarrow\bm{V}$ is self adjoint. Therefore 
the orthogonal complement of $\bm{u}$ is a $T$-invariant subspace denoted by $\bm{U}^{\perp,\bm{V}}$.
Moreover the eigenvalue $\lambda$ does not belong to the spectrum 
of $T|_{{\bm{U}^{\perp,\bm{V}}}}$ which is defined as 
$T|_{{\bm{U}^{\perp,\bm{V}}}}: {\bm{U}^{\perp,\bm{V}}}\rightarrow {\bm{U}^{\perp,\bm{V}}}$. 
Thus, we can define the following invertible map
\begin{align*}
    (I-\lambda T): {\bm{U}^{\perp,\bm{V}}}\rightarrow {\bm{U}^{\perp,\bm{V}}}.
\end{align*}
Here $(I-\lambda T)^{-1}$ is bounded.
Define $\bm{\delta}_h:= \bm{u}_{\lambda,h}-\bm{u}_h-(\bm{u}_{\lambda,h}-\bm{u}_h,\bm{u})\bm{u}$.
Using $\|\bm{u}\|_0=1$, it follows
\begin{align*}
    (\bm{\delta}_h,\bm{u})=(\bm{u}_{\lambda,h}-\bm{u}_h-(\bm{u}_{\lambda,h}-\bm{u}_h,\bm{u})\bm{u},\bm{u})
    =0.
\end{align*}
Hence we have
\begin{align*}
    \bm{\delta}_h\in {\bm{U}^{\perp,\bm{V}}}\quad\mbox{and}\quad \|\bm{\delta}_h\|_0
    \lesssim \|(I-\lambda T)\bm{\delta}_h\|_0.
\end{align*}
Since $(I-\lambda T)\bm{u}=0$, the following estimate holds
\begin{align}\label{mpost110}
    \|\bm{\delta}_h\|_0\lesssim \|(I-\lambda T)(\bm{u}_{\lambda,h}-\bm{u}_h)\|_0.
\end{align}
Moreover, we obtain the following inequality 
\begin{align}\label{mpost11}
\begin{split}
    &(I-\lambda T)(\bm{u}_{\lambda,h}-\bm{u}_h)\\
    &\quad=(\lambda_h T_h-\lambda T)(\bm{u}_{\lambda,h}-\bm{u}_h)+\bm{u}_{\lambda,h}-\bm{u}_h
    -\lambda_h T_h(\bm{u}_{\lambda,h}-\bm{u})-\lambda_h T_h(\bm{u}-\bm{u}_h)\\
    &\quad=(\lambda_h T_h-\lambda T)(\bm{u}_{\lambda,h}-\bm{u}_h)+(\lambda-\lambda_h)T_h\bm{u}
    -\lambda_h T_h(\bm{u}_{\lambda,h}-\bm{u}).
\end{split}
\end{align} 
Inserting the result of \eqref{mpost11} into \eqref{mpost110} and using the triangle inequality, implies
\begin{align*}
    \|\bm{\delta}_h\|_0
    &\lesssim \|(\lambda_hT_h-\lambda T_h)(\bm{u}_{\lambda,h}-\bm{u}_h)\|_0
    +\|(\lambda T_h-\lambda T)(\bm{u}_{\lambda,h}-\bm{u}_h)\|_0\\
    &\quad+\|(\lambda-\lambda_h)T_h\bm{u}\|_0
    +\|\lambda_h T_h(\bm{u}_{\lambda,h}-\bm{u})\|_0,\\
    &\lesssim \mathcal{I}_1+\mathcal{I}_2+\mathcal{I}_3+\mathcal{I}_4,
\end{align*}
where 
\begin{align*}
    \mathcal{I}_1&=\|(\lambda_h T_h-\lambda T_h)(\bm{u}_{\lambda,h}-\bm{u}_h)\|_0,\quad
    \mathcal{I}_2=\|(\lambda T_h-\lambda T)(\bm{u}_{\lambda,h}-\bm{u}_h)\|_0,\\
    \mathcal{I}_3&=\|(\lambda-\lambda_h)T_h\bm{u}\|_0,\quad\quad\quad\quad\quad\quad\;\;
    \mathcal{I}_4=\|\lambda_h T_h(\bm{u}_{\lambda,h}-\bm{u})\|_0.
\end{align*}
Using the boundedness of $T_h$, the estimates of $\mathcal{I}_1$ and $\mathcal{I}_3$ read
\begin{align}\label{eq:I13}
    \mathcal{I}_1&\lesssim |\lambda-\lambda_h| \|\bm{u}_{\lambda,h}-\bm{u}_h\|_0 \quad\mbox{and}\quad
    \mathcal{I}_3\lesssim |\lambda-\lambda_h|.
\end{align}
Using $\bm{f}=(\bm{u}_{\lambda,h}-\bm{u}_h)$ in \eqref{eqsoei13} and inserting the regularity result \eqref{soreg11}, we obtain
\begin{align}\label{eq:I2}
    \mathcal{I}_2=\|(\lambda T_h-\lambda T)(\bm{u}_{\lambda,h}-\bm{u}_h)\|_0\lesssim \lambda h\|(\bm{u}_{\lambda,h}-\bm{u}_h)\|_0.
\end{align}
Adding and subtracting $T_h(\bm{P}_h\bm{u})$ in $\mathcal{I}_4$ leads to
\begin{align}\label{lemph11}
    \mathcal{I}_4\leq \lambda_h \|T_h(\bm{u}_{\lambda,h}-\bm{P}_h\bm{u})\|_0
    +\lambda_h \|T_h(\bm{P}_h\bm{u}-\bm{u})\|_0.
\end{align}
Applying Lemma \ref{mposterr12} in first term of the estimate \eqref{lemph11}, implies
\begin{align*}
    \|T_h(\bm{u}_{\lambda,h}-\bm{P}_h\bm{u})\|_0\lesssim h^{k+3}(\|\bm{\sigma}\|_{k+2}+\|\bmdiv\bm{\sigma}\|_{k+1}).
\end{align*} 
Note that the second term in the estimate \eqref{lemph11} is equal to zero due to definition of $T_h$,
since the right hand side of 
problem \eqref{nmsprob131} vanishes for $\bm{f}= \bm{P}_h\bm{u}-\bm{u}$. 
Hence 
\begin{align*}
    \mathcal{I}_4\lesssim \lambda_h h^{k+3}(\|\bm{\sigma}\|_{k+2}+\|\bmdiv\bm{\sigma}\|_{k+1}).
\end{align*} 
Using the definition of $\bm{\delta}_h$, leads to
\begin{align}\label{deltapost12}
    \|\bm{u}_{\lambda,h}-\bm{u}_h\|_0\leq \|\bm{\delta}_h\|_0+\|(\bm{u}_{\lambda,h}-\bm{u}_h,\bm{u})\bm{u}\|_0.
\end{align}
The first term of the estimate \eqref{deltapost12} is already estimated. 
To estimate the second term, observe that with $\|\bm{u}\|_0=1$, we get
\begin{align*}
    \|(\bm{u}_{\lambda,h}-\bm{u}_h,\bm{u})\bm{u}\|_0=|(\bm{u}_{\lambda,h}-\bm{u}_h,\bm{u})|.
\end{align*} 
Moreover, we have
\begin{align*}
    \|(\bm{u}_{\lambda,h}-\bm{u}_h,\bm{u})\bm{u}\|_0\leq |(\bm{u}_{\lambda,h}-\bm{u},\bm{u})|+|(\bm{u}-\bm{u}_h,\bm{u})|.
\end{align*}
Since $\|\bm{u}\|_0=1$ and $\|\bm{u}_h\|_0=1$, we get
\begin{align}\label{eq:J1}
    |(\bm{u}-\bm{u}_h,\bm{u})| = 1-(\bm{u}_h,\bm{u})  = \frac{1}{2}\|\bm{u}-\bm{u}_h\|_0^2 \lesssim h^{2k+2}.
\end{align}
Next, we estimate $|(\bm{u}_{\lambda,h}-\bm{u},\bm{u})|$. 
Choosing the test functions $\bm{\tau}=\bm{\sigma}_{\lambda,h}-\bm{\sigma}$ and
$\bm{v}=\bm{u}_{\lambda,h}-\bm{u}$ in \eqref{prob12}, gives
\begin{align*}
    (\mathcal{A}\bm{\sigma},\bm{\sigma}_{\lambda,h}-\bm{\sigma})+(\bmdiv  (\bm{\sigma}_{\lambda,h}-\bm{\sigma}),\bm{u})&=0,\\
    (\bmdiv  \bm{\sigma}, \bm{u}_{\lambda,h}-\bm{u})& =-\lambda (\bm{u},\bm{u}_{\lambda,h}-\bm{u}).
\end{align*}
From that, we obtain
\begin{align*}
    -\lambda(\bm{u}_{\lambda,h}-\bm{u},\bm{u})=(\bmdiv  \bm{\sigma}, \bm{u}_{\lambda,h}-\bm{u})
    +(\mathcal{A}\bm{\sigma},\bm{\sigma}_{\lambda,h}-\bm{\sigma})+(\bmdiv  (\bm{\sigma}_{\lambda,h}-\bm{\sigma}),\bm{u}).
\end{align*}
Moreover, we have
\begin{align}\label{newdispost13}
\begin{split}
    -\lambda(\bm{u}_{\lambda,h}-\bm{u},\bm{u})
    &=(\bmdiv  (\bm{\sigma}-\bm{\sigma}_{\lambda,h}), \bm{u}_{\lambda,h}-\bm{u})+(\bmdiv  \bm{\sigma}_{\lambda,h}, \bm{u}_{\lambda,h}-\bm{u})\\
    &\quad+(\mathcal{A}(\bm{\sigma}-\bm{\sigma}_{\lambda,h}),\bm{\sigma}_{\lambda,h}-\bm{\sigma})+(\mathcal{A}\bm{\sigma}_{\lambda,h},\bm{\sigma}_{\lambda,h}-\bm{\sigma})\\
    &\quad+(\bmdiv  (\bm{\sigma}_{\lambda,h}-\bm{\sigma}),\bm{u}-\bm{u}_{\lambda,h})+(\bmdiv  (\bm{\sigma}_{\lambda,h}-\bm{\sigma}),\bm{u}_{\lambda,h}).
\end{split}
\end{align}
Choosing $\bm{\tau}_h=\bm{\sigma}_{\lambda,h}$, $\bm{v}_h=\bm{u}_{\lambda,h}$, $\bm{f}=\lambda\bm{u}$ in \eqref{nmsprob131} and  $\bm{\tau}=\bm{\sigma}_{\lambda,h}$, 
$\bm{v}=\bm{u}_{\lambda,h}$ in \eqref{prob12}, and subtracting \eqref{prob12} from \eqref{nmsprob131}, 
we obtain
\begin{align}\label{newdispost12}
\begin{split}
    (\mathcal{A}(\bm{\sigma}_{\lambda,h}-\bm{\sigma}),\bm{\sigma}_{\lambda,h})+(\bmdiv  \bm{\sigma}_{\lambda,h},\bm{u}_{\lambda,h}-\bm{u})&=0,\\
    (\bmdiv  (\bm{\sigma}_{\lambda,h}-\bm{\sigma}), \bm{u}_{\lambda,h})& =0.
\end{split}
\end{align}
Inserting the result from \eqref{newdispost12} in \eqref{newdispost13}, implies
\begin{align}\label{eq:J2}
\begin{split}
    -\lambda(\bm{u}_{\lambda,h}-\bm{u},\bm{u})&=(\mathcal{A}(\bm{\sigma}-\bm{\sigma}_{\lambda,h}),\bm{\sigma}_{\lambda,h}-\bm{\sigma})
    +2(\bmdiv(\bm{\sigma}-\bm{\sigma}_{\lambda,h}), \bm{u}_{\lambda,h}-\bm{u})\\
    &\lesssim \|\bm{\sigma}-\bm{\sigma}_{\lambda,h}\|^2_0
    +\|\bmdiv(\bm{\sigma}-\bm{\sigma}_{\lambda,h})\|^2_0+\|\bm{u}-\bm{u}_{\lambda,h}\|_0^2\\
    &\lesssim h^{2k+4}\|\bm{\sigma}\|_{k+2}^2+h^{2k+2}\|\bmdiv\bm{\sigma}\|_{k+1}^2 + h^{2k+2}\|\bm{u}\|_{k+2}^2\\
    &\lesssim h^{2k+2}.
\end{split}
\end{align}
Summing up the estimates \eqref{eq:I13}--\eqref{lemph11} and \eqref{eq:J1}, \eqref{eq:J2}, we obtain
\begin{align*}
    \|\bm{u}_{\lambda,h}-\bm{u}_h\|_0
    &\lesssim |\lambda-\lambda_h| \|\bm{u}_{\lambda,h}-\bm{u}_h\|_0
    + h\|\bm{u}_{\lambda,h}-\bm{u}_h\|_0 
    + |\lambda-\lambda_h|\\
    &\quad+h^{k+3}(\|\bm{\sigma}\|_{k+2}+\|\bmdiv \bm{\sigma}\|_{k+1})
    + h^{2k+2}.
\end{align*}
Choosing $h$ small enough and using $|\lambda-\lambda_h|\lesssim h^{2k+2}$, the following estimate holds
\begin{align*}
    \|\bm{u}_{\lambda,h}-\bm{u}_h\|_0\lesssim h^{k+3}(h^{k-1} + \|\bm{\sigma}\|_{k+2}+\|\bmdiv \bm{\sigma}\|_{k+1}).
\end{align*}
\end{proof}
\begin{proof}[Proof of Theorem \ref{errthm11}]
Combining Lemma \ref{mposterr12} and \ref{mposterr12311}, implies the desired result.
\end{proof}
%
\section{Local post-processing}\label{secpost}
In this section we present an improved approximation of the eigenfunction 
by a local post-processing similar to the one in \cite{CCJGEJP} for the source problem. 
Define for $m\ge k+2$
\begin{align*}
    \bm{W}^*_h=\{\bm{v}\in \bm{L}^2(\Omega)\;|\;\bm{v}|_{K}\in P_m(K;\mathbb{R}^2)\;\mbox{for all}\;K\in\mathcal{T}_h\}.
\end{align*}
Choose $\bm{u}_h^*\in\bm{W}^*_h $ on each $K\in\mathcal{T}_h$ with $\bm{P}_K=\bm{P}_h|_K$ 
as a solution of the system
\begin{align}\label{postpro11}
\begin{split}
    \bm{P}_K\bm{u}_h^*&=\bm{u}_h,\\
    (\bm{\epsilon}(\bm{u}_h^*),\bm{\epsilon}(\bm{v}))_K&=(\mathcal{A}\bm{\sigma}_h,\bm{\epsilon}(\bm{v}))_K
    \quad\mbox{for all}\;\bm{v}\in(\bm{\delta}-\bm{P}_K) \bm{W}^*_h|_K.
\end{split}
\end{align}
Here, $\bm{u}_h^*$ is the Riesz representation of the linear functional $(\mathcal{A}\bm{\sigma}_h,\bm{\epsilon}(\cdot))_K$ in the 
Hilbert space 
\[
(\bm{\delta}-\bm{P}_K) \bm{W}^*_h|_K
\equiv\{\bm{v}_m\in P_m(K;\mathbb{R}^2)| (\bm{v}_m,\bm{w}_k)_K=0 \;\mbox{for all} \;\{\bm{w}_k\in P_k(K;\mathbb{R}^2) \},
\]
which is equipped with the scalar product $(\bm{\epsilon}(\cdot),\bm{\epsilon}(\cdot))_K$.
We define the local post-processing on each triangle with Lagrange multiplier 
$\bm{\mu}_k\in P_k(K;\mathbb{R}^2)$
as follows
\begin{align}\label{postiden11}
\begin{split}
    (\bm{\epsilon}(\bm{u}_h^*),\bm{\epsilon}(\bm{v}_m))_K+(\bm{\mu}_k,\bm{v}_m)_K
    &=(\mathcal{A}\bm{\sigma}_h,\bm{\epsilon}(\bm{v}_m))_K\quad\mbox{for all}\;\bm{v}_m\in P_m(K;\mathbb{R}^2),\\
    (\bm{u}_h^*,\bm{w}_k)_K&=(\bm{u}_h,\bm{w}_k)_K\quad\quad\;\;\mbox{for all}\;\bm{w}_k\in P_k(K;\mathbb{R}^2).
\end{split}
\end{align}  
From Korn's inequality, we get positive definiteness of $(\bm{\epsilon}(\cdot),\bm{\epsilon}(\cdot))_K$ 
on $(\bm{\delta}-\bm{P}_K) \bm{W}^*_h|_K$, and 
since $P_k(K;\mathbb{R}^2)\subset P_m(K;\mathbb{R}^2)$, we obtain 
\begin{align*}
    \sup_{0\neq\bm{v}_{m}\in P_m(K;\mathbb{R}^2)}\frac{(\bm{v}_m,\bm{\mu}_k)_K}{\|\bm{v}_m\|_{1,K}}\ge C\|\bm{\mu}_k\|_{0,K}
    \quad\mbox{for all}\;\bm{\mu}_k\in P_k(K;\mathbb{R}^2).
\end{align*}
Hence, there exist unique solutions on each triangle, c.f. \cite{FB}. Combining the identity 
$\mathcal{A}\bm{\sigma}=\bm{\epsilon}(\bm{u})$ and \eqref{postpro11}, gives the following error identity
\begin{align*}
    (\bm{\epsilon}(\bm{u}-\bm{u}_h^*),\bm{\epsilon}(\bm{v}))_K
    =(\mathcal{A}(\bm{\sigma}-\bm{\sigma}_h),\bm{\epsilon}(\bm{v}))_K\quad\mbox{for all}\;\bm{v}\in(\bm{\delta}-\bm{P}_K) \bm{W}^*_h|_K.
\end{align*}
\begin{theorem}\label{posterror11}
With sufficiently smooth boundary $\partial \Omega$, if $\bm{u}\in\bm{H}^{m+1}(\Omega)$,  
$\bm{\sigma}\in\bm{H}^{k+2}(\Omega;\mathbb{S})$, and $\bmdiv  \bm{\sigma}\in \bm{H}^{k+1}(\Omega)$ 
solve problem \eqref{pre12}, then the following estimates hold for the post-processed eigenfunction 
$\bm{u}^*_h\in\bm{W}^*_h$
\begin{align*}
    \|\bm{u}-\bm{u}_h^*\|_0
    &\lesssim h^{k+3}(h^{k-1}+\|\bm{\sigma}\|_{k+2}+\|\bmdiv\bm{\sigma}\|_{k+1}) +h^{m+1}\|\bm{u}\|_{m+1},\\
    \|\nabla_{\mathcal{T}_h}(\bm{u}-\bm{u}_h^*)\|_0
    &\lesssim h^{k+2}(h^{k-1}+\|\bm{\sigma}\|_{k+2}+\|\bmdiv\bm{\sigma}\|_{k+1}) +h^{m}\|\bm{u}\|_{m+1}.
\end{align*}
\end{theorem}
\begin{proof}
Let $\hat{\bm{u}}$ denote the $\bm{L}^{2}$ projection of $\bm{u}$ onto $\bm{W}^{*}_h$. 
From the triangle inequality we get
\begin{align}\label{eq:thm4.1:rhs}
    \|\bm{u}-\bm{u}_h^*\|_0\leq \|\bm{u}-\hat{\bm{u}}\|_0
    +\|\bm{P}_h(\hat{\bm{u}}-\bm{u}_h^*)\|_0
    +\|(\bm{\delta}-\bm{P}_h)(\hat{\bm{u}}-\bm{u}_h^*)\|_0.
\end{align}
Applying \eqref{serr15} in first part of the right hand side of \eqref{eq:thm4.1:rhs} leads to
\begin{align}\label{perrthm1}
    \|\bm{u}-\hat{\bm{u}}\|_0\lesssim h^{m+1}\|\bm{u}\|_{m+1}\quad\mbox{for all}\quad \bm{u}\in\bm{H}^{m+1}(\Omega).
\end{align}
Using $\bm{P}_{K}\bm{u}^*_h=\bm{u}_h$ on each $K\in\mathcal{T}_h$ shows 
\begin{align*}
    \|\bm{P}_{K}(\hat{\bm{u}}-\bm{u}_h^*)\|_{0,K}=\|\bm{P}_{K}\hat{\bm{u}}-\bm{u}_h\|_{0,K}.
\end{align*}  
Here $\bm{V}_{h}\subset \bm{W}^*_h$, so we can apply the result of Theorem \ref{errthm11}
to deduce
\begin{align}\label{perrthm2}
    \|\bm{P}_{h}(\hat{\bm{u}}-\bm{u}_h^*)\|_0 = \|\bm{P}_{h}\bm{u}-\bm{u}_h\|_0\lesssim h^{k+3}(h^{k-1}+\|\bm{\sigma}\|_{k+2}+\|\bmdiv\bm{\sigma}\|_{k+1}).
\end{align} 
The estimates for the last term in \eqref{eq:thm4.1:rhs}
are identical to the estimates for the post-processing for
the source problem, hence we omit the details and quote the final estimate
from the proof of \cite[Theorem $4.1$]{CCJGEJP}
\begin{align}\label{perrthm3}
    \|(\bm{\delta}-\bm{P}_h)(\hat{\bm{u}}-\bm{u}_h^*)\|_0
    \lesssim h^{k+3}(\|\bm{\sigma}\|_{k+2}+\|\bmdiv\bm{\sigma}\|_{k+1})+h^{m+1}\|\bm{u}
\|_{m+1}.
\end{align}
Combining \eqref{eq:thm4.1:rhs} with \eqref{perrthm1}--\eqref{perrthm3} proves the first estimate.
\par
To prove the second estimate, we use the triangle inequality 
\begin{align*}
    |\bm{u}-\bm{u}_h^*|_{1,K}\leq |\bm{u}-\hat{\bm{u}}|_{1,K}+|\bm{P}_h(\hat{\bm{u}}-\bm{u}_h^*)|_{1,K}
    +|(\bm{\delta}-\bm{P}_h)(\hat{\bm{u}}-\bm{u}_h^*)|_{1,K}.
\end{align*}
Applying an discrete inverse inequality leads to
\begin{align*}
    |\bm{u}-\bm{u}_h^*|_{1,K}\leq |\bm{u}-\hat{\bm{u}}|_{1,K}
    +h^{-1}\|\bm{P}_h(\hat{\bm{u}}-\bm{u}_h^*)\|_{0,K}
    +h^{-1}\|(\bm{\delta}-\bm{P}_h)(\hat{\bm{u}}-\bm{u}_h^*)\|_{0,K}.
\end{align*}
Combining \eqref{perrthm2} and \eqref{perrthm3}, with the estimate
\begin{align*}
|\bm{u}-\hat{\bm{u}}|_{1,K} \lesssim h^m \|\bm{u}\|_{m+1,K}
\end{align*}
obtained from interpolation type arguments and inverse estimates as in the proof of \cite[Theorem $4.1$]{CCJGEJP},
proves the second estimate.  
\end{proof}
\par
\begin{definition}
For $(\bm{u}^*_h,\bm{u}^*_h)\neq 0$, we define the post-processed eigenvalue as the value of the 
Rayleigh quotient of the post-processed eigenfunction 
\begin{align}\label{postrayle11}
    \lambda_h^*:=-\frac{(\bmdiv \bm{\sigma}_h, \bm{u}^*_h)}{(\bm{u}^*_h,\bm{u}^*_h)}.
\end{align}
\end{definition}
\begin{theorem}
Let  $(\bm{\sigma},\bm{u},\lambda)\in (\bm{\Phi}\cap\bm{H}^{k+2}(\Omega;\mathbb{S}))\times\bm{H}^{k+2}(\Omega)\times \mathbb{R}_+$ 
be the solution of \eqref{prob12}, with $\|\bm{u}\|_0=1$.
For sufficiently small $h$,
the following a priori estimate 
for the post-processed eigenvalue  $\lambda_h^*$ holds
\begin{align*}
    |\lambda-\lambda_h^*|\lesssim h^{2k+4}.
\end{align*}
\end{theorem}
\begin{proof}
From  \eqref{prob12}, \eqref{mfpro11}, \eqref{postiden11}, \eqref{postrayle11}, and 
$(\bm{u},\bm{u})=1$, we have
\begin{align*}
    (2\nu\mathcal{A}\bm{\sigma},\mathcal{A}\bm{\sigma})=\lambda(\bm{u},\bm{u})=\lambda,
    \quad\textrm{and}\quad
    (2\nu\mathcal{A}\bm{\sigma}_h,\mathcal{A}\bm{\sigma}_h)=\lambda_h^*(\bm{u}_h^*,\bm{u}_h^*).
\end{align*}
Moreover, it follows
\begin{align*}
    2\nu\|\mathcal{A}(\bm{\sigma}-\bm{\sigma}_h)\|_{0}^2
    &=2\nu(\mathcal{A}(\bm{\sigma}-\bm{\sigma}_h),\mathcal{A}(\bm{\sigma}-\bm{\sigma}_h))\\
    &=2\nu(\mathcal{A}\bm{\sigma},\mathcal{A}\bm{\sigma})+2\nu(\mathcal{A}\bm{\sigma}_h,\mathcal{A}\bm{\sigma}_h)
    -4\nu(\mathcal{A}\bm{\sigma},\mathcal{A}\bm{\sigma}_h)\\
    &=\lambda+\lambda_h^*(\bm{u}_h^*,\bm{u}_h^*)-2\lambda_h^*(\bm{u},\bm{u}_h^*)+2(\bmdiv\bm{\sigma}_h,\bm{u})
    +2\lambda_h^*(\bm{u},\bm{u}_h^*).
\end{align*}
Therefore we have
\begin{align*}
    \lambda- \lambda_h^*
    =2\nu\|\mathcal{A}(\bm{\sigma}-\bm{\sigma}_h)\|_{0}^2-\lambda_h^*-\lambda_h^*(\bm{u}^*_h, \bm{u}^*_h)
    +2\lambda_h^*(\bm{u},\bm{u}_h^*)-2(\bmdiv\bm{\sigma}_h+\lambda_h^*\bm{u}_h^*,\bm{u}).
\end{align*}
Using 
$(\bm{u},\bm{u})+(\bm{u}^*_h,\bm{u}^*_h)-2(\bm{u},\bm{u}_h^*)=\|\bm{u}-\bm{u}_h^*\|^2_{0}$,
and $(\bmdiv\bm{\sigma}_h+\lambda_h^*\bm{u}_h^*,\bm{u}_h^*)=0$
we get
\begin{align}\label{eq:p4.3}
    \lambda- \lambda_h^*=
    &2\nu\|\mathcal{A}(\bm{\sigma}-\bm{\sigma}_h)\|_{0}^2-\lambda_h^*\|\bm{u}-\bm{u}_h^*\|^2_{0}
    -2(\bmdiv\bm{\sigma}_h+\lambda_h^*\bm{u}_h^*,\bm{u}-\bm{u}_h^*),
\end{align}
which leads to
\begin{align*}
    \lambda- \lambda_h^*=
    &2\nu\|\mathcal{A}(\bm{\sigma}-\bm{\sigma}_h)\|_{0}^2-\lambda_h^*\|\bm{u}-\bm{u}_h^*\|^2_{0}
    -2(\bmdiv\bm{\sigma}_h-\bmdiv\bm{\sigma},\bm{u}-\bm{u}_h^*)\\
    &-2(\lambda_h^*\bm{u}_h^*-\lambda\bm{u},\bm{u}-\bm{u}_h^*)\\
    =&2\nu\|\mathcal{A}(\bm{\sigma}-\bm{\sigma}_h)\|_{0}^2-\lambda_h^*\|\bm{u}-\bm{u}_h^*\|^2_{0}
    -2(\bmdiv\bm{\sigma}_h-\bmdiv\bm{\sigma},\bm{u}-\bm{u}_h^*)\\
    &-2(\lambda_h^*(\bm{u}_h^*-\bm{u}),\bm{u}-\bm{u}_h^*)-2((\lambda_h^*-\lambda)\bm{u},\bm{u}-\bm{u}_h^*)\\
    =&2\nu\|\mathcal{A}(\bm{\sigma}-\bm{\sigma}_h)\|_{0}^2+\lambda_h^*\|\bm{u}-\bm{u}_h^*\|^2_{0}
    +2(\bmdiv(\bm{\sigma}-\bm{\sigma}_h),\bm{u}-\bm{u}_h^*)\\
    &+2(\lambda-\lambda_h^*)(\bm{u},\bm{u}-\bm{u}_h^*).
\end{align*}
From Cauchy-Schwarz inequality it follows
\begin{align*}
    |\lambda-\lambda_h^*|&\lesssim 
    2\nu\|\mathcal{A}(\bm{\sigma}-\bm{\sigma}_h)\|^{2}_{0}+\lambda_h^*\|\bm{u}-\bm{u}_h^*\|^{2}_{0}
    +2\|\bmdiv(\bm{\sigma}-\bm{\sigma}_h)\|_0\|\bm{u}-\bm{u}_h^*\|_0\\
    &\quad+2|\lambda-\lambda_h^*| \|\bm{u}-\bm{u}_h^*\|_0.
\end{align*} 
This results in
\begin{align*}
    |\lambda-\lambda_h^*|\lesssim \|\bm{\sigma}-\bm{\sigma}_h\|^{2}_{0}
    +\|\bm{u}-\bm{u}_h^*\|^{2}_{0}+\|\bmdiv(\bm{\sigma}-\bm{\sigma}_h)\|_0\|\bm{u}-\bm{u}_h^*\|_0
    +|\lambda-\lambda_h^*|^2.
\end{align*} 
Using  the estimates \eqref{erreigs12} and \eqref{erreigdiv}, and 
Theorem~\ref{posterror11}, we obtain for $h$ small enough
\begin{align*}
    |\lambda-\lambda_h^*|\lesssim h^{2k+4}.
\end{align*}
\end{proof}
%
\section{A posteriori error analysis}\label{postera}
First, we define $\tilde{\bm{u}}_h\in \bm{H}=\bm{H}^{1}_0(\Omega)$. 
Here $\tilde{\bm{u}}_h$
is closely related to the discontinuous approximation $\bm{u}_h^*$. 
\par
The primal mixed formulation of the Stokes source 
problem with right hand side $\lambda\bm{u}$ reads as follows: 
find $\bm{\sigma}\in\bm{L}$ and  $\bm{u}\in \bm{H}$ such that
\begin{align}\label{posprob11}
\begin{split}
    -(\bm{\sigma}, \bm{\epsilon}(\bm{v}))& =-\lambda (\bm{u},\bm{v})\quad\mbox{for all}\;\bm{v}\in{\bm{H}},\\
    (2\nu\mathcal{A}\bm{\sigma},\mathcal{A}\bm{\tau})-(\bm{\tau},\bm{\epsilon}(\bm{u}))&=0
    \qquad\qquad\;\,\mbox{for all}\;\bm{\tau}\in{\bm{L}}.
\end{split}
\end{align}
Following \cite[Section 1.3]{CC}, we have the following lemma
\begin{lemma}
The operator $A:X\rightarrow X^*$, defined for $(\bm{\sigma},\bm{u})\in X:= \bm{L}\times \bm{H }$ by
\begin{align*}
    (A(\bm{\sigma},\bm{u}))(\bm{\tau},\bm{v}):= 
    (2\nu\mathcal{A}\bm{\sigma},\mathcal{A}\bm{\tau})-(\bm{\sigma}, \bm{\epsilon}(\bm{v}))-(\bm{\tau},\bm{\epsilon}(\bm{u})),
\end{align*}
is linear, bounded, and bijective.
\end{lemma}
From the above lemma, we obtain 
\begin{align}\label{eq:resequiv}
    \|\mathcal{A}(\bm{\sigma}-\bm{\sigma}_h)\|_0+\|\bm{\epsilon}(\bm{u}-\tilde{\bm{u}}_h)\|_0 
    \approx\|Res_L\|_{\bm{L}^*} + \|Res_H\|_{\bm{H}^*},
\end{align}
for any approximation $(\bm{\sigma}_h,\tilde{\bm{u}}_h)\in \bm{L}\times \bm{H }$
of the primal source problem with right hand side $\lambda \bm{u}$,
where 
\begin{align*}
    Res_H(\bm{v})&:= {\lambda}({\bm{u}},\bm{v})+(\bm{\sigma}_h,\bm{\epsilon}({\bm{v}}))\qquad\quad\;\mbox{for all}\;\bm{v}\in{\bm{H}},\\
    Res_L(\bm{\tau})&:= (2\nu\mathcal{A}\bm{\sigma}_h,\mathcal{A}\bm{\tau})-(\bm{\tau},\bm{\epsilon}(\tilde{\bm{u}}_h))
    \quad\mbox{for all}\;\bm{\tau}\in{\bm{L}}.
\end{align*}
\begin{lemma}\label{posprob12}
Let $(\bm{\sigma},\bm{u},\lambda)\in \bm{L}\times \bm{H}\times \mathbb{R}_+$ be a solution of \eqref{prob12}. 
Then the approximation 
$(\bm{\sigma}_h,\tilde{\bm{u}}_h)\in \bm{L}\times \bm{H}$ 
satisfies 
\begin{align*}
    &\|\mathcal{A}(\bm{\sigma}-\bm{\sigma}_h)\|_0+\|\bm{\epsilon}_{\mathcal{T}_h}(\bm{u}-\tilde{\bm{u}}_h)\|_0\\
    &\qquad\lesssim 
    \|\mathcal{A}{\bm{\sigma}}_h-\bm{\epsilon}(\tilde{\bm{u}}_h)\|_{K,0}
    +\Big(\sum_{K\in\mathcal{T}_h}h_{K}^2\|\lambda \bm{u}+\bmdiv\bm{\sigma}_h\|^2_0\Big)^{1/2}
    +\bm{\Theta},
\end{align*}
with the higher order term
\begin{align*}
    \bm{\Theta} &:= \lambda_h \| \bm{u} - \bm{u}_h^*\|_0 + |\lambda-\lambda_h|.
\end{align*}
\end{lemma}
\begin{proof}
Gauss theorem implies for any $\bm{v}\in\bm{H}$
\begin{align*}
    Res_H(\bm{v})&= ({\lambda} {\bm{u}}+\bmdiv\bm{\sigma}_h,\bm{v}).
\end{align*}
Let $\bm{v}_h$ denote the Scott-Zhang interpolation \cite{SZ} of $\bm{v}$, then it holds that
\begin{align*}
    Res_H(\bm{v})=(\lambda\bm{u}+\bmdiv\bm{\sigma}_h,\bm{v}-\bm{v}_h) 
    +(\lambda\bm{u}+\bmdiv\bm{\sigma}_h,\bm{v}_h).
\end{align*}
For the second term on the right hand side, \eqref{mfpro11} and \eqref{postpro11} show
\begin{align*}
(\lambda\bm{u}+\bmdiv\bm{\sigma}_h,\bm{v}_h)
&= (\lambda\bm{u}-\lambda_h\bm{u}_h,\bm{v}_h)
= (\lambda\bm{u}-\lambda_h\bm{u}^*_h,\bm{v}_h)\\
&= (\lambda - \lambda_h)(\bm{u},\bm{v}_h) + \lambda_h (\bm{u}-\bm{u}_h^*,\bm{v}_h).
\end{align*}
Applying Cauchy-Schwarz inequality, it follows
\begin{align*}
    Res_H(\bm{v})&\lesssim 
    \left(\sum_{K\in\mathcal{T}_h}h_{K}^2\|\lambda\bm{u}+\bmdiv\bm{\sigma}_h\|^2_{K,0}\right)^{1/2}
    \left(\sum_{K\in\mathcal{T}_h}h_{K}^{-2}\|(\bm{v}-\bm{v}_h) \|^2_{K,0}\right)^{1/2}\\
    &\quad+\left( \lambda_h \| \bm{u} - \bm{u}_h^*\| + |\lambda-\lambda_h| \right)
    \|\bm{v}_h \|_0.
\end{align*}
Poincare's inequality, and stability and approximation properties of the Scott-Zhang interpolation
show
\begin{align*}
    Res_H(\bm{v})\lesssim \left(\left(\sum_{K\in\mathcal{T}_h}h_{K}^2\|\lambda\bm{u}
    +\bmdiv\bm{\sigma}_h\|^2_{K,0}\right)^{1/2}
    +\bm{\Theta}\right)|\bm{v}|_1.
\end{align*}
Finally, $\|Res_L\|_{\bm{L}^*}$ is estimated as
\begin{align*}
    Res_L(\bm{\tau})=
    \int_{\Omega}(\mathcal{A} \bm{\sigma}_h-\bm{\epsilon}(\tilde{\bm{u}}_h)):
    \bm{\tau} dx\le \|\mathcal{A} \bm{\sigma}_h-\bm{\epsilon}(\tilde{\bm{u}}_h)\|_0
    \|\bm{\tau}\|_0.
\end{align*}
\end{proof}
Now, we present an a posteriori error estimator that
involves the discontinuous post-processed approximation 
\[ 
\bm{u}_h^*\in H^{1}(\mathcal{T}_h) :=\{\bm{v}\in \bm{L}^2(\Omega)| \bm{v}|_T\in \bm{H}^1(T)\;\mbox{for all} \;T\in\mathcal{T}_h\}.
\]
In the following let $\tilde{\bm{u}}_h \in \bm{H}$ be the conforming approximation
to $\bm{u}$,
that is obtained from the  discontinuous post-processed function $\bm{u}_h^*$
by taking the arithmetic mean value
\begin{align*}
\tilde{\bm{u}}_h(z):=\frac{1}{|\{K\in \mathcal{T}_h : z\in K\}|}\sum_{K\in \mathcal{T}_h : z\in K} \bm{u}_h^*(z)|_K
\end{align*}
for each vertex and edge degree of freedom in $z\in \mathbb{R}^2$.
A discrete scaling argument and \cite[Theorem 2.2]{KP} show
\begin{align}\label{karakashian}
\sum_{K\in\mathcal{T}_h} h_{K}^{-2} \| \bm{u}_h^* - \tilde{\bm{u}}_h\|_{K,0}^2
+ \sum_{K\in\mathcal{T}_h} \| \bm{\epsilon}_{\mathcal{T}_h}(\bm{u}_h^*) - \bm{\epsilon}(\tilde{\bm{u}}_h)\|_{K,0}^2
\lesssim
\sum_{E\in \mathcal{E}_h}h_{E}^{-1}\|[\bm{u}_h^*]\|_{E,0}^2.
\end{align}
\begin{theorem}\label{mpost11n11}
Let $(\bm{\sigma},\bm{u},\lambda)\in \bm{L}\times\bm{H}\times\mathbb{R}_+$ be a solution of \eqref{prob12}. 
The post-processed eigenfunction $\bm{u}_{h}^{*}\in H^{1}(\mathcal{T}_h)$ satisfies the following
reliability estimate
\begin{align}\label{mpost121}
    \|\mathcal{A}(\bm{\sigma}-\bm{\sigma}_h)\|_0+\|\bm{\epsilon}_{\mathcal{T}_h}(\bm{u}-\bm{u}_h^*)\|_0
    &\lesssim \eta+\bm{\Upsilon},
\end{align}
for the a~posteriori error estimator
\begin{align*}
    \eta^2 :=  \|\mathcal{A}\bm{\sigma}_h-\bm{\epsilon}_{\mathcal{T}_h} (\bm{u}_h^*)\|^2_0
    +\sum_{K\in\mathcal{T}_h}h_{K}^2\|\lambda^*_h \bm{u}^*_h+\bmdiv\bm{\sigma}_h\|^2_{K,0}
    +\sum_{E\in \mathcal{E}_h}h_{E}^{-1}\|[\bm{u}_h^*]\|^2_{E,0},
\end{align*}
and the higher order term
\begin{align*}    
    \bm{\Upsilon} := 
    \left(\sum_{T\in\mathcal{T}_h}h_{K}^2\|\lambda\bm{u} - \lambda_h^{*}\bm{u}_h^*\|^2_{K,0}\right)^{1/2}
    +\bm{\Theta}.
\end{align*} 
\end{theorem}
\begin{proof}
Adding and subtracting $\bm{\epsilon}(\tilde{\bm{u}}_h)$ in the second term of \eqref{mpost121} 
and  using the triangle inequality, we have
\begin{align*}
    \|\mathcal{A}(\bm{\sigma}-\bm{\sigma}_h)\|_0+\|\bm{\epsilon}_{\mathcal{T}_h}(\bm{u}-\bm{u}_h^*)\|_0
    &\lesssim \|\mathcal{A}(\bm{\sigma}-\bm{\sigma}_h)\|_0+\|\bm{\epsilon}_{\mathcal{T}_h}(\bm{u}-\tilde{\bm{u}}_h)\|_0\\
    &\quad+\|\bm{\epsilon}_{\mathcal{T}_h}(\bm{u}_h^*)-\bm{\epsilon}(\tilde{\bm{u}}_h)\|_0.
\end{align*}  
Applying Theorem \ref{posprob12}, implies
\begin{align*}
    \|\mathcal{A}(\bm{\sigma}-\bm{\sigma}_h)\|_0+\|\bm{\epsilon}_{\mathcal{T}_h}(\bm{u}-\bm{u}_h^*)\|_0
    &\lesssim\|\mathcal{A} \bm{\sigma}_h-\bm{\epsilon}(\tilde{\bm{u}}_h)\|_0
    +\|\bm{\epsilon}_{\mathcal{T}_h}(\bm{u}_h^*)-\bm{\epsilon}(\tilde{\bm{u}}_h)\|_0\\
    &\quad+\left(\sum_{K\in\mathcal{T}_h}h_{K}^2\|\lambda\bm{u}+\bmdiv\bm{\sigma}_h\|^2_{K,0}\right)^{1/2}
    +\bm{\Theta}.
\end{align*}  
Another triangle inequality yields
\begin{align*}
    \|\mathcal{A}(\bm{\sigma}-\bm{\sigma}_h)\|_0+\|\bm{\epsilon}_{\mathcal{T}_h}(\bm{u}-\bm{u}_h^*)\|_0
    &\lesssim\|\mathcal{A} \bm{\sigma}_h-\bm{\epsilon}_{\mathcal{T}_h}(\bm{u}_h^*)\|_0
    +\|\bm{\epsilon}_{\mathcal{T}_h}(\bm{u}_h^*)-\bm{\epsilon}(\tilde{\bm{u}}_h)\|_0\\
    &\quad+\left(\sum_{K\in\mathcal{T}_h}h_{K}^2\|\lambda^*_h \bm{u}^*_h+\bmdiv\bm{\sigma}_h\|^2_{K,0}\right)^{1/2}
    +\bm{\Upsilon}.
\end{align*}
Using \eqref{karakashian}, the desired estimate holds.
\end{proof}
\begin{theorem}
The a~posteriori error estimator $\eta^2$
provides an upper bound of the
post-processed  eigenvalue error 
for sufficiently small mesh size h, up to the
higher order term $\bm{\Upsilon}^2$,
\begin{align*}
|\lambda-\lambda^{*}_h|\lesssim \eta^2+\bm{\Upsilon}^2.
\end{align*}
\end{theorem}
\begin{proof}
Adding and subtracting $\tilde{\bm{u}}_h$ in the last term of \eqref{eq:p4.3} yields
\begin{align*}
\lambda- \lambda_h^*&=2\nu\|\mathcal{A}(\bm{\sigma}-\bm{\sigma}_h)\|_{0}^2
-\lambda_h^*\|\bm{u}-\bm{u}_h^*\|^2_{0}-2(\bmdiv\bm{\sigma}_h
+\lambda_h^*\bm{u}_h^*,\bm{u}-\tilde{\bm{u}}_h)\\
&\quad-2(\bmdiv\bm{\sigma}_h+\lambda_h^*\bm{u}_h^*,\tilde{\bm{u}}_h-\bm{u}_h^*)\\
&=2\nu\|\mathcal{A}(\bm{\sigma}-\bm{\sigma}_h)\|_{0}^2
-\lambda_h^*\|\bm{u}-\bm{u}_h^*\|^2_{0}
-2(\bm{\mbox{div}}\bm{\sigma}_h-\bm{\mbox{div}}\bm{\sigma},\bm{u}-\tilde{\bm{u}}_h)\\
&\quad-2(\lambda_h^*\bm{u}_h^*-\lambda\bm{u},\bm{u}-\tilde{\bm{u}}_h)
-2(\bmdiv\bm{\sigma}_h+\lambda_h^*\bm{u}_h^*,\tilde{\bm{u}}_h-\bm{u}_h^*).
\end{align*}
Applying Gauss theorem, implies
\begin{align*}
\lambda- \lambda_h^*
&=2\nu\|\mathcal{A}(\bm{\sigma}-\bm{\sigma}_h)\|_{0}^2+\lambda_h^*\|\bm{u}-\bm{u}_h^*\|^2_{0}-2(\bm{\sigma}-\bm{\sigma}_h,\nabla(\bm{u}-\tilde{\bm{u}}_h))\\
&\quad+ 2\lambda_h^*(\bm{u}-\bm{u}_h^*,\bm{u}-\tilde{\bm{u}}_h) +2(\lambda-\lambda_h^*)(\bm{u},\bm{u}-\tilde{\bm{u}}_h)\\
&\quad-2(\bm{\mbox{div}}\bm{\sigma}_h+\lambda_h^*\bm{u}_h^*,\tilde{\bm{u}}_h-\bm{u}_h^*).
\end{align*} 
Using Cauchy-Schwarz inequality, we have
\begin{align*}
|\lambda- \lambda_h^*|
&\lesssim\|\mathcal{A}(\bm{\sigma}-\bm{\sigma}_h)\|_{0}^2+\|\bm{u}-\bm{u}_h^*\|^2_{0}
+ \|\bm{\sigma}-\bm{\sigma}_h\|_{0}^2
+\|\nabla(\bm{u}-\tilde{\bm{u}}_h)\|_0^2
+\|\bm{u}-\tilde{\bm{u}}_h\|_0^2\\
&\quad+ (\lambda-\lambda_h^*)^2
+\sum_{K\in\mathcal{T}_h}h_K^2\|\lambda_h^*\bm{u}_h^*+\bmdiv\bm{\sigma}_h\|^2_{0,K}
+ \sum_{K\in\mathcal{T}_h}h_K^{-2}\|\tilde{\bm{u}}_h-\bm{u}_h^*\|_{0,K}.
\end{align*}
From the trace estimate \cite[Proposition 3.1, IV.3]{BF}, we get
\begin{align*}
  \|\bm{\sigma}-\bm{\sigma}_h\|_{0} \lesssim \|\mathcal{A}(\bm{\sigma}-\bm{\sigma}_h)\|_{0} + \|\bmdiv(\bm{\sigma}-\bm{\sigma}_h)\|_{-1}.
\end{align*}
Note that
\begin{align*}
\|\bmdiv(\bm{\sigma}-\bm{\sigma}_h)\|_{-1} 
= \sup_{\substack{\bm{v}\in \bm{H}\\ |\bm{v}|_1=1}} | (\lambda\bm{u}+\bmdiv\bm{\sigma}_h,\bm{v}) |
= \|Res_H\|_{\bm{H}^*}.
\end{align*}
Therefore, \eqref{eq:resequiv}, Korn's and triangle inequalities lead to
\begin{align*}
|\lambda- \lambda_h^*| &\lesssim
\|\mathcal{A}(\bm{\sigma}-\bm{\sigma}_h)\|_{0}^2
+ \|\bm{\epsilon}_{\mathcal{T}_h}(\bm{u}-\bm{u}_h^*)\|_0^2
+\sum_{K\in\mathcal{T}_h}h_K^2\|\lambda_h^*\bm{u}_h^*+\bmdiv\bm{\sigma}_h\|^2_{0,K} \\
&\quad +\|\bm{u}_h^*-\tilde{\bm{u}}_h\|^2_{0}+\|\bm{\epsilon}_{\mathcal{T}_h}(\bm{u}_h^*-\tilde{\bm{u}}_h)\|_0^2
+\sum_{K\in\mathcal{T}_h}h_K^{-2}\|\bm{u}_h^*-\tilde{\bm{u}}_h\|^2_{0,K}\\
&\quad +\|\bm{u}-\bm{u}_h^*\|^2_{0} + (\lambda-\lambda_h^*)^2.
\end{align*}
For $h$ small enough such that $|\lambda-\lambda_h^*| < 1/2$, we conclude with \eqref{karakashian}
and \eqref{mpost121}
\begin{align*}
|\lambda-\lambda_h^*|&\lesssim\eta^2+\bm{\Upsilon}^2.
\end{align*}
\end{proof}
%
\section{Numerical experiments}\label{comre}
In this section, we present numerical results
for the square domain, the L-shaped domain and the slit domain.
We verify the proven (asymptotic) reliability of the a~posteriori
error estimator of Section~\ref{postera} and 
show empirically its efficiency.
We present numerical results that show sixth order
convergence of the post-processed eigenvalues of Section~\ref{secpost} on adaptively
refined meshes even for the (nonconvex) L-shaped and slit domains.
In all experiments, we take to the lowest order Arnold-Winther finite element $(k=1)$, 
the parameter $\nu=1$ and the polynomial order $P_3(\mathcal{T}_h, \mathbb{R}^2)$ 
for the post-processing. 
\par
Since the exact eigenvalues for all three domains are unknown, we
compare the computed eigenvalues to some reference values with high
accuracy. Note that the eigenvalues of the Stokes eigenvalue
problem are related to the eigenvalues of the buckling eigenvalue problem of clamped plates
via the stream function formulation.
Hence, we can use known \cite{BPEBPT} or computed reference eigenvalues for
the plate eigenvalue problem.
\par
We consider the standard adaptive finite element loop
\begin{align*}
\mbox{Solve}\rightarrow\mbox{Estimate}\rightarrow\mbox{Mark}\rightarrow\mbox{Refine},
\end{align*}  
that creates a sequence of adaptively refined (nested) regular meshes $(\mathcal{T_\ell})$
level index $\ell$.
For the algebraic eigenvalue solver we use the Matlab implementation of ARPACK \cite{ARPACK}.
To estimate the error, we compute the a~posteriori error estimator of Section~\ref{postera}
\begin{align*}
\eta_\ell^2&=\sum_{K\in\mathcal{T}_\ell}\|\mathcal{A}\bm{\sigma}_\ell-\bm{\epsilon}_{\mathcal{T}_\ell} (\bm{u}_\ell^*)\|_{0,K}^2+
\sum_{E\in \mathcal{E}_\ell}h_{E}^{-1}\|[\bm{u}_\ell^*]\|_{0,E}^2+\sum_{K\in\mathcal{T}_\ell}h_{K}^2\|\lambda^*_\ell \bm{u}^*_\ell+
\bmdiv\bm{\sigma}_\ell\|_{0,K}^2.
\end{align*}
In the above formula, $\bm{u}^*_\ell\in \bm{W}_\ell^*$ is the solution of the local post-processing 
which is given in Section \ref{secpost}.
Since the conforming function $\tilde{\bm{u}}_\ell$ of Section~\ref{postera} is closely related to $\bm{u}_\ell^*$
we also compare the error estimator $\eta_\ell^2$ to the heuristical estimator
\begin{align*}
\mu_\ell^2&=\sum_{K\in\mathcal{T}_\ell}\|\mathcal{A}\bm{\sigma}_\ell-\bm{\epsilon}(\tilde{\bm{u}}_\ell)\|_{0,K}^2+
\sum_{K\in\mathcal{T}_\ell}h_{K}^2\|\tilde{\lambda}_\ell \tilde{\bm{u}}_\ell+\bmdiv\bm{\sigma}_\ell\|_{0,K}^2,
\end{align*}
where we replaced $\bm{u}_\ell^*$ by $\tilde{\bm{u}}_\ell$, and
$\tilde{\lambda}_\ell$ is computed from \eqref{postrayle11}
with $\bm{u}_\ell^*$ replaced by $\tilde{\bm{u}}_\ell$.
We numerically demonstrate reliability and efficiency of $\mu_\ell^2$ for the eigenvalue error
$|\lambda-\tilde{\lambda}_\ell|$.
We mark triangles of the triangulation $\mathcal{T}_\ell$ in a minimal set 
of marked triangles $\mathcal{M}_\ell$ according to the 
bulk marking strategy \cite{Doerfler}, such that $\theta\eta_\ell^2\leq \eta_\ell^2(\mathcal{M}_\ell)$ 
for the bulk parameter $\theta= 1/2$, and refine the mesh
with the red-green-blue refinement strategy \cite{RV}.
\par
Let $N_\ell$ denote the degrees of freedom
$N_\ell:=\mbox{dim}(\bm{\Phi}_\ell)+\mbox{dim}(\bm{V}_\ell)$.
Note that for uniform meshes,  we have the relationship $\mathcal{O} (N_\ell^{-r})\approx \mathcal{O}(h_\ell^{2r}),\; r>0$.
%
%
\begin{figure}[tbp]
\centering
\subfigure[]{
\includegraphics[width = \figurewidtha\textwidth]{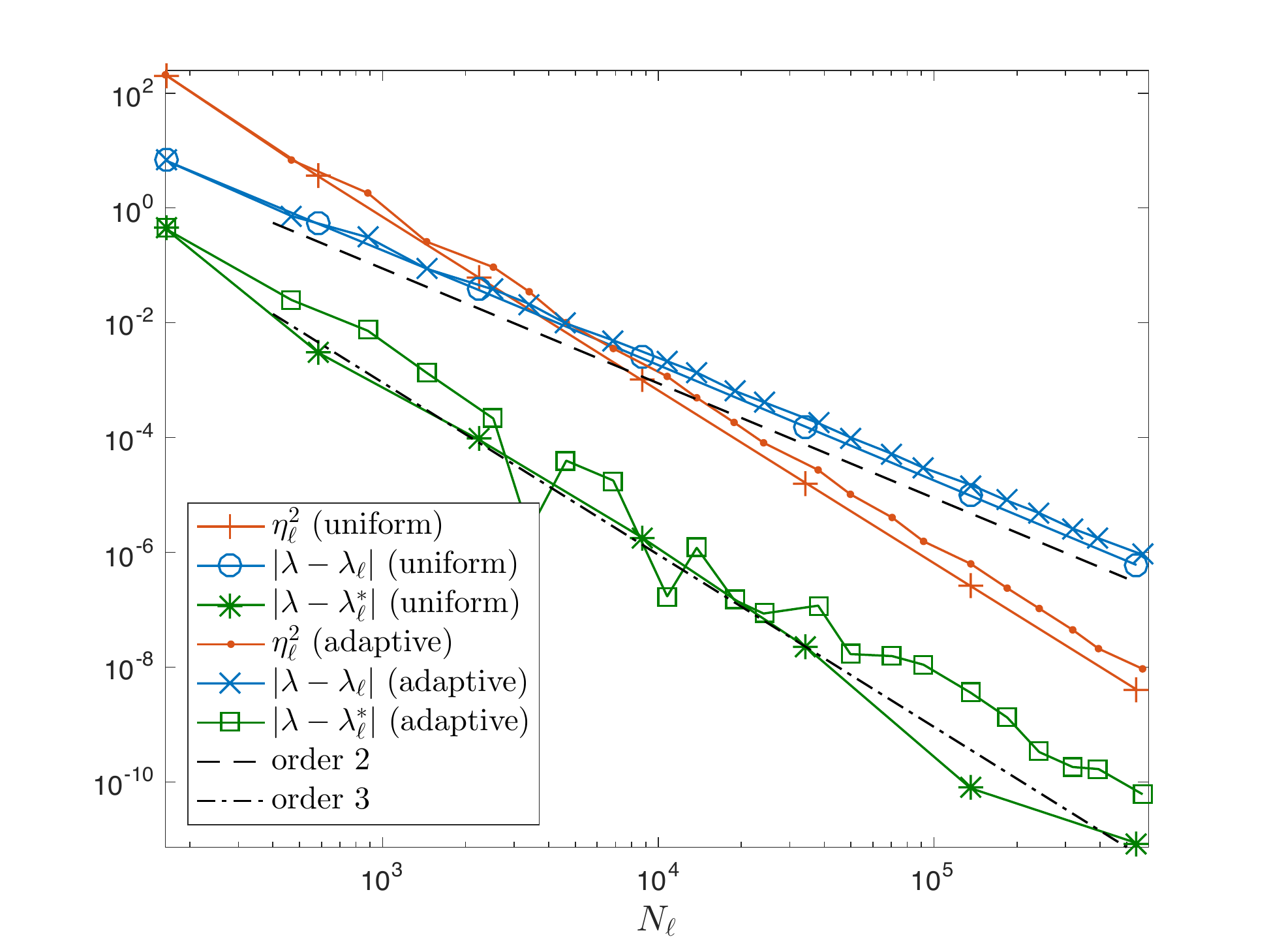}
\label{figex21}
}
\subfigure[]{
\includegraphics[width = \figurewidtha\textwidth]{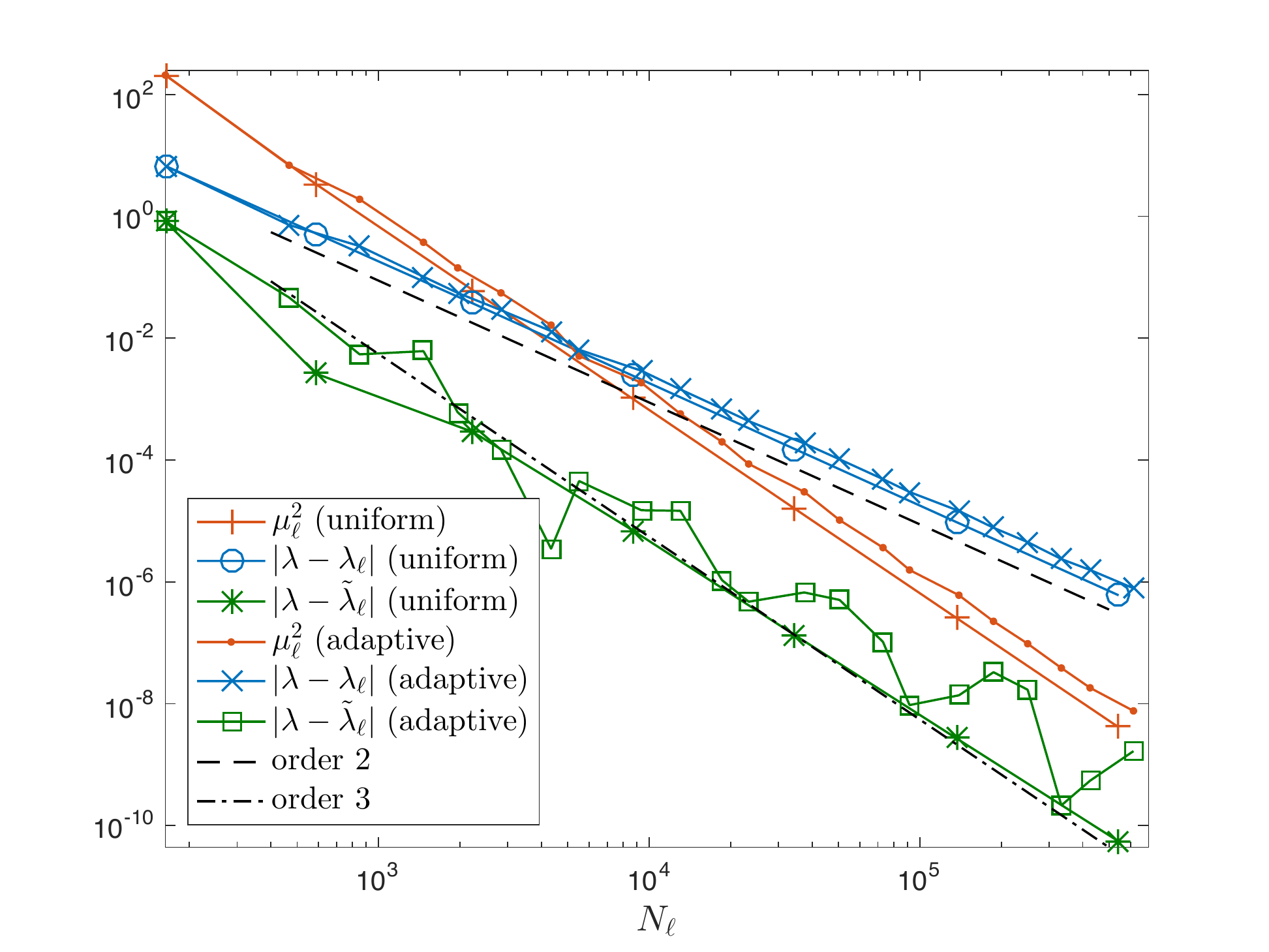}
\label{figex21mu}
}
\caption[]
{Convergence history of 
\subref{figex21}  $|\lambda-\lambda_{\ell}|, |\lambda-\lambda_{\ell}^*|$, $\eta_{\ell}^2$, \subref{figex21mu}  
$|\lambda-\lambda_{\ell}|, |\lambda-\tilde{\lambda}_{\ell}|$ and $\mu_{\ell}^2$ on uniformly and adaptively 
refined meshes for the square domain.}
\label{figex21aw}
\end{figure}
\begin{figure}[tbp]
\centering
\subfigure[]{
\includegraphics[width = \figurewidthb\textwidth]{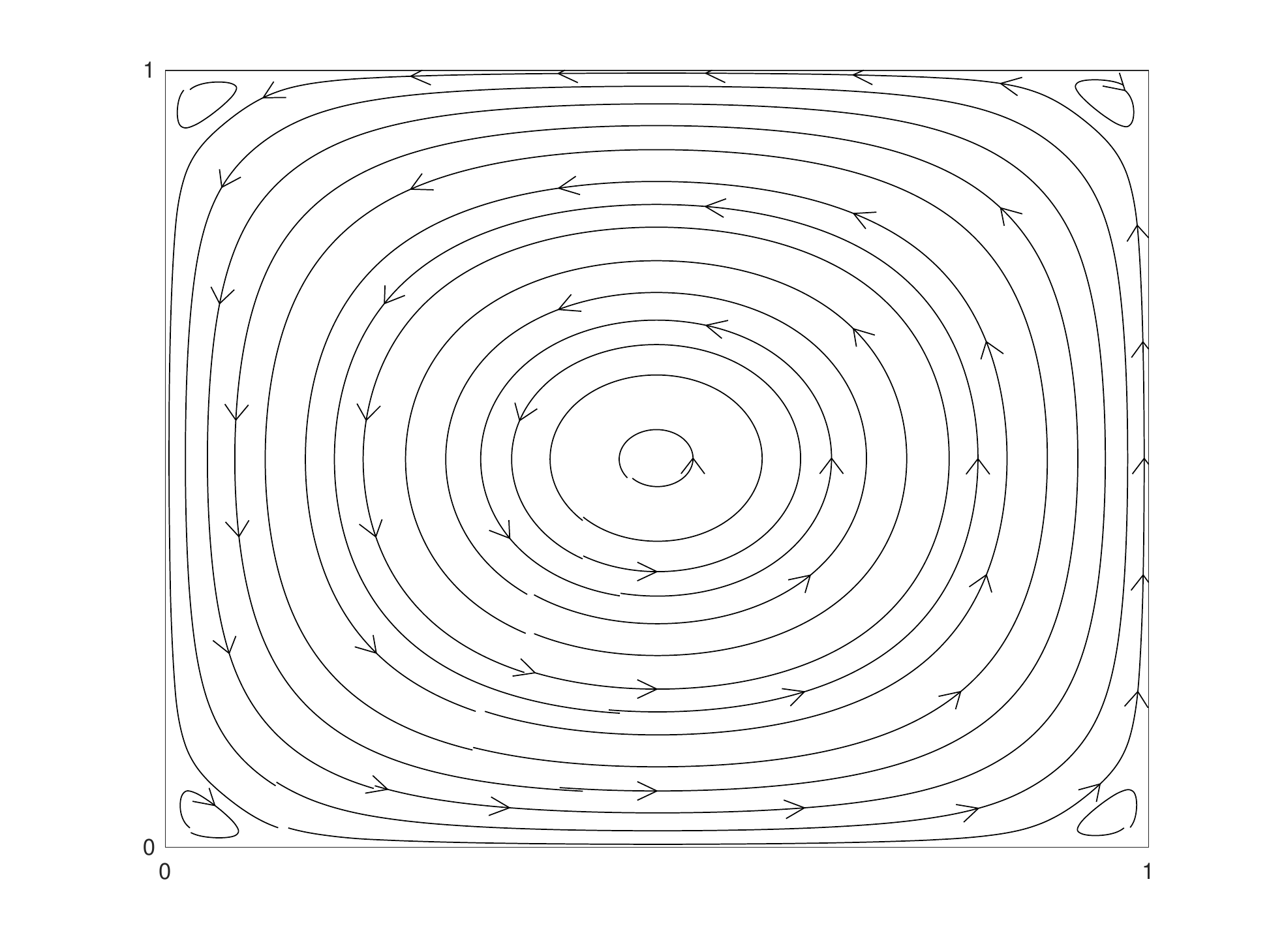}
\label{fig22a}
}
\subfigure[]{
\includegraphics[width = \figurewidthb\textwidth]{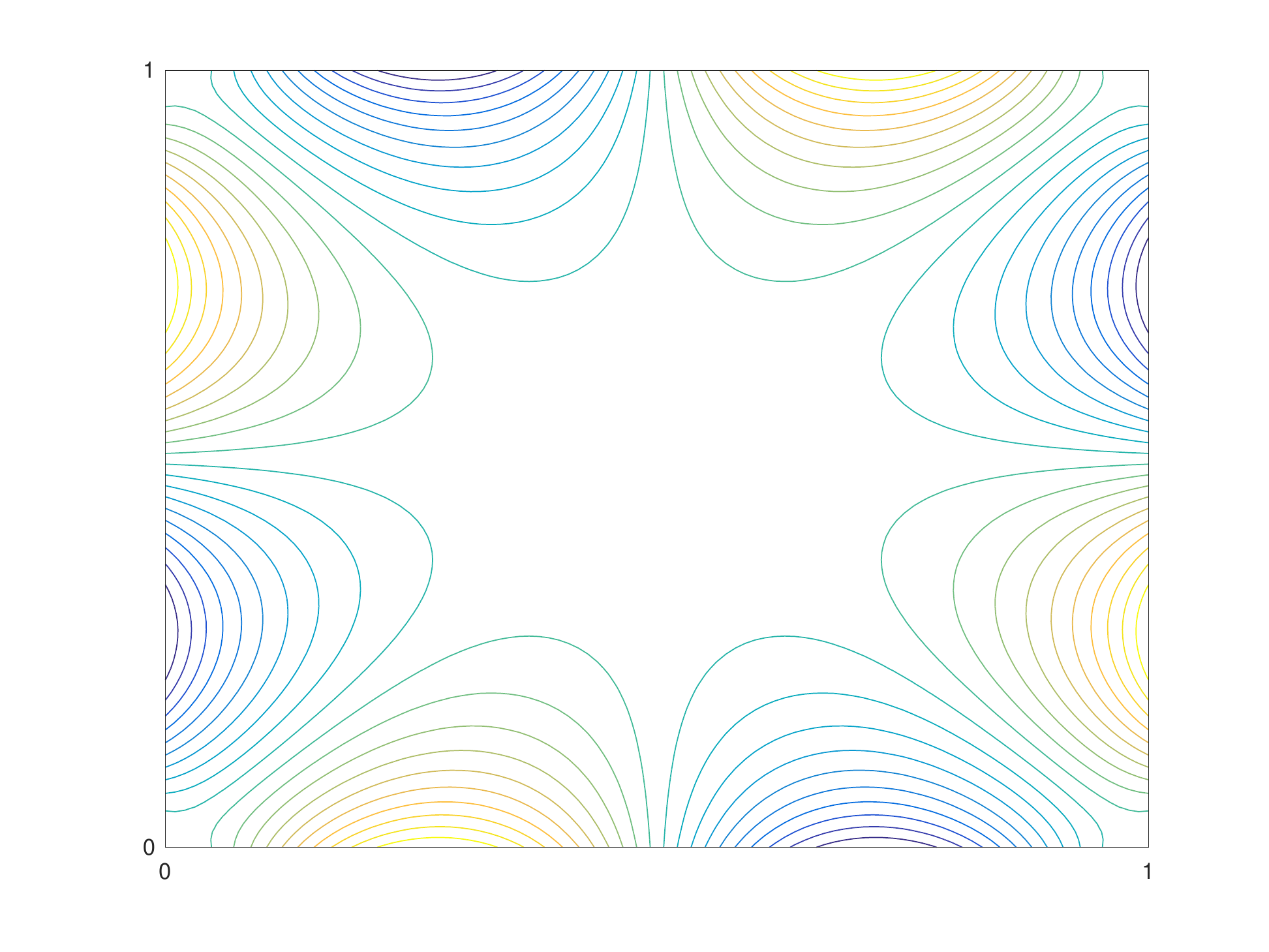}
\label{fig22c}
}
\caption[]
{\subref{fig22a} Streamline plot of the discrete eigenfunction $\bm{u}_{\ell}$.
\subref{fig22c} Plot of discrete pressure $p_{\ell}=-\mbox{tr}\bm{\sigma}_{\ell}/2$.}
\label{figex212}
\end{figure}
\subsection{Square domain}\label{testsquared}
In the first example, we consider the square domain $\Omega=(0,1)^2$.  
The reference value for the first eigenvalue 
$\lambda=52.344691168$ is taken from \cite{BPEBPT,CWQL}.
Figures \ref{figex21} and \ref{figex21mu} are devoted to the convergence history 
of the eigenvalue errors and a~posteriori error estimators.
Due to the smoothness of the eigenfunction,
the error of the post-processed eigenvalue
$\lambda_{\ell}^*$ and the corresponding a~posteriori error
estimator $\eta_\ell^2$ achieve optimal third order 
of convergence for both uniform and adaptive meshes. 
Moreover, the error for the eigenvalue $\tilde{\lambda}_{\ell}$ and 
the estimator $\mu_\ell^2$ also achieve optimal convergence of $\mathcal{O} (N^{-3})$
for both uniform and adaptive meshes. For both uniform and 
adaptive meshes, the convergence rate of 
the eigenvalue error of $\lambda_{\ell}$ is equal to $\mathcal{O} (N^{-2})$, which confirms the theoretical result 
\eqref{eigenlam12}. 
The streamline plot of the discrete eigenfunction $\bm{u}_{\ell}$ and 
a plot of the discrete pressure $p_{\ell}=-\mbox{tr}\bm{\sigma}_{\ell}/2$ 
are displayed in \ref{fig22a} and \ref{fig22c}, respectively. 
%
\subsection{L-shaped domain}\label{testlsha}
\begin{figure}[tbp]
\centering
\subfigure[]{
\includegraphics[width = \figurewidtha\textwidth]{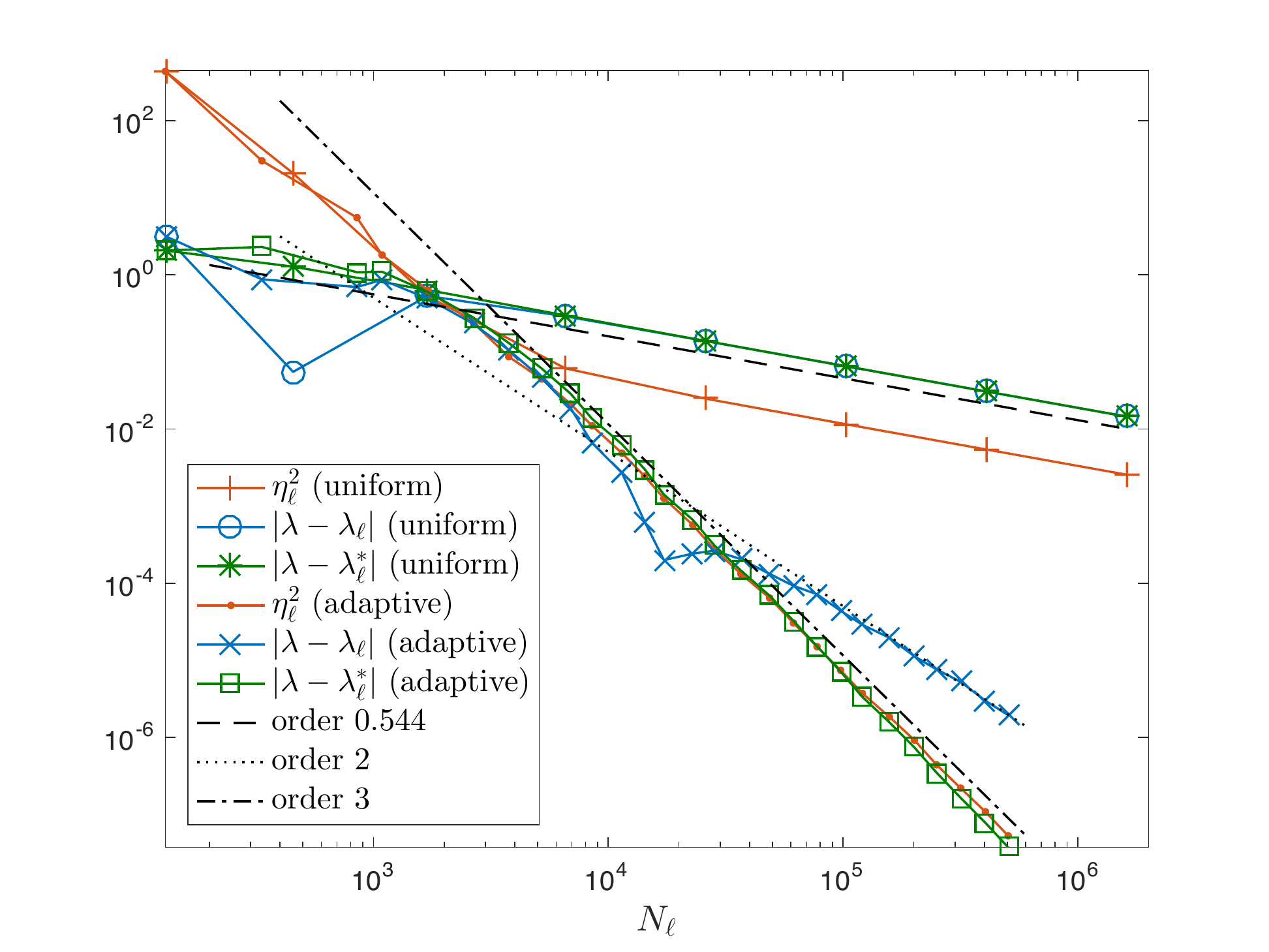}
\label{figLsh122a}
}
\subfigure[]{
\includegraphics[width = \figurewidtha\textwidth]{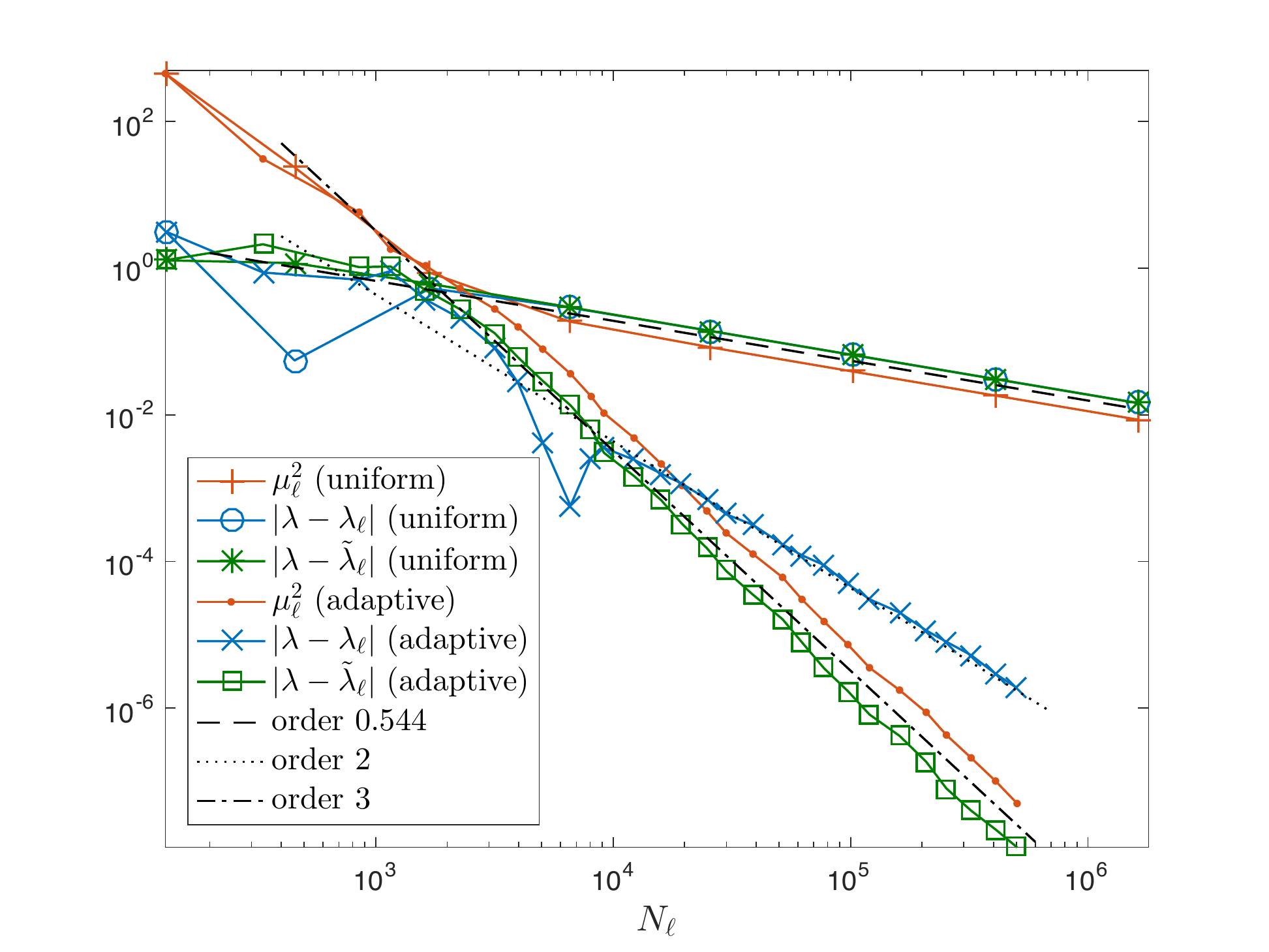}
\label{figLsh122c}
}
\caption[]
{Convergence history of  \subref{figLsh122a}  $|\lambda-\lambda_{\ell}|, |\lambda-\lambda_{\ell}^*|$, $\eta_{\ell}^2$   
\subref{figLsh122c} $|\lambda-\lambda_{\ell}|, |\lambda-\tilde{\lambda}_{\ell}|$ 
and $\mu_{\ell}^2$ on uniformly and adaptively refined meshes for the L-shaped domain.}
\label{figexLsh21}
\end{figure}
\begin{figure}[tbp]
\centering
\subfigure[]{
\includegraphics[width = \figurewidthb\textwidth]{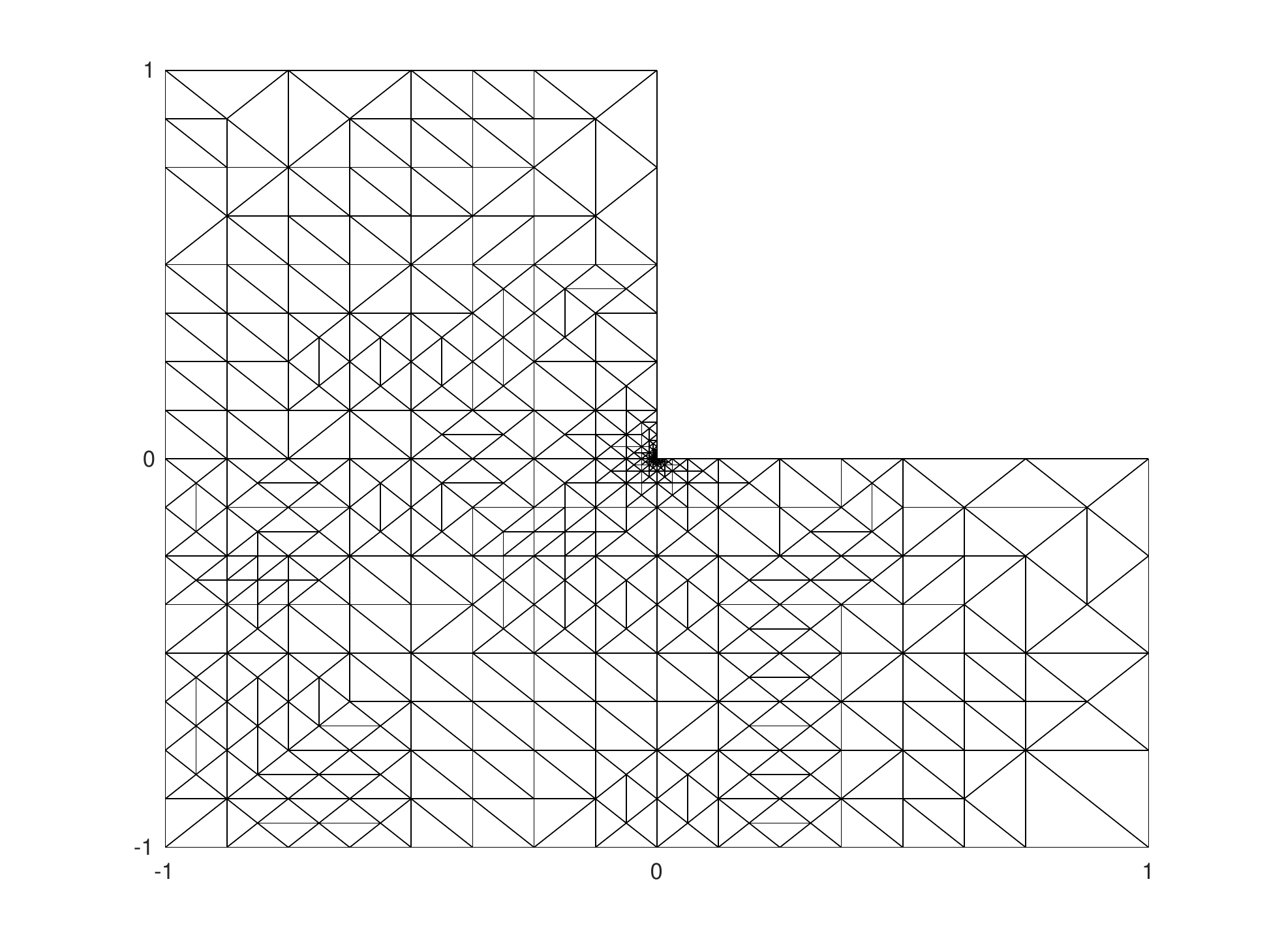}
\label{figLsh122aw}
}
\subfigure[]{
\includegraphics[width = \figurewidthb\textwidth]{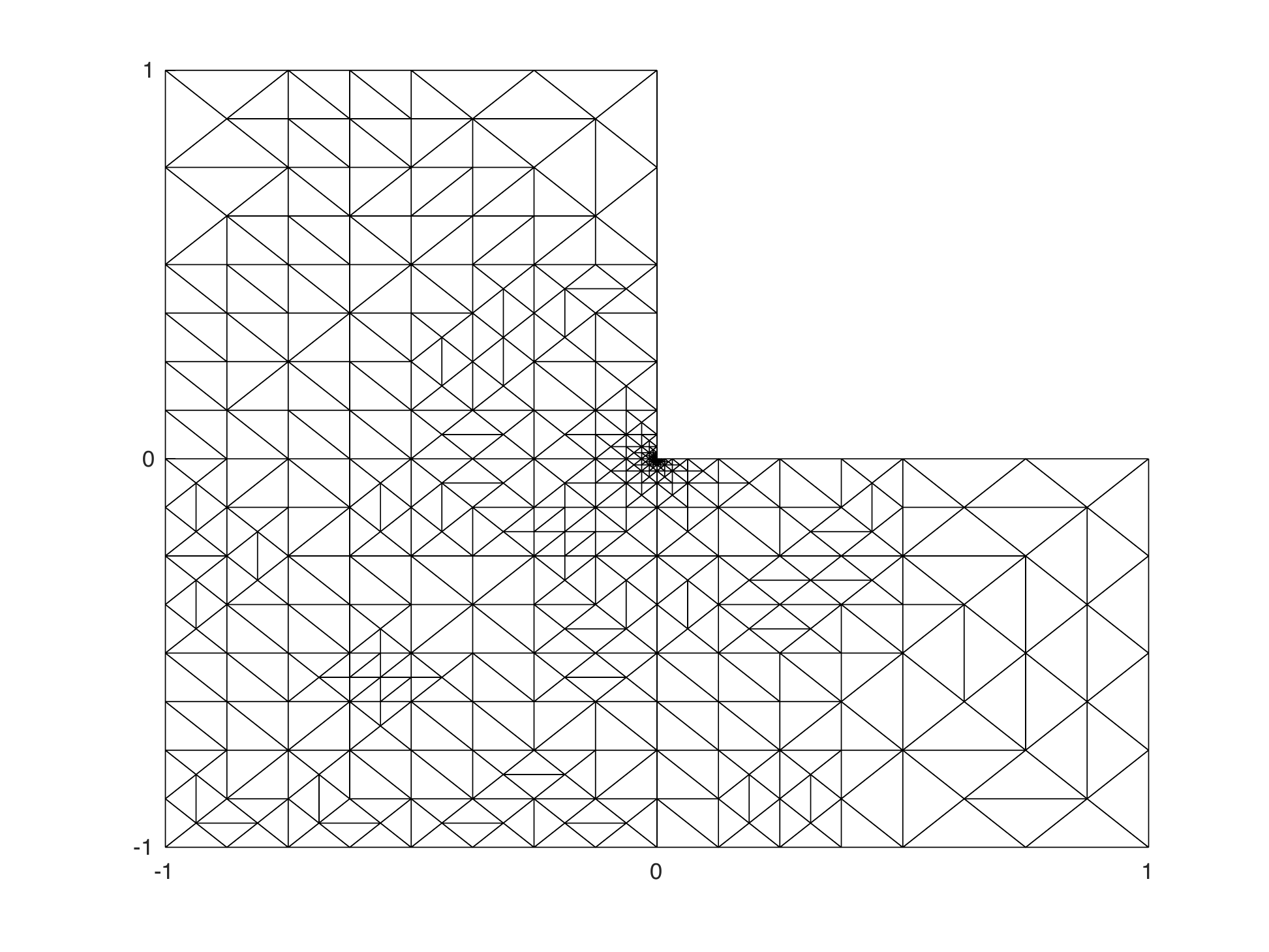}
\label{figLsh122cw}
}
\caption[]
{\subref{figLsh122a} Adaptive refined meshes for $\eta_{\ell}^2$ with $386$ nodes, 
\subref{figLsh122c} Adaptively refined meshes for $\mu_{\ell}^2$ with $392$ nodes.}
\label{figexLsh21w}
\end{figure}
\begin{figure}[tbp]
\centering
\subfigure[]{
\includegraphics[width = \figurewidthb\textwidth]{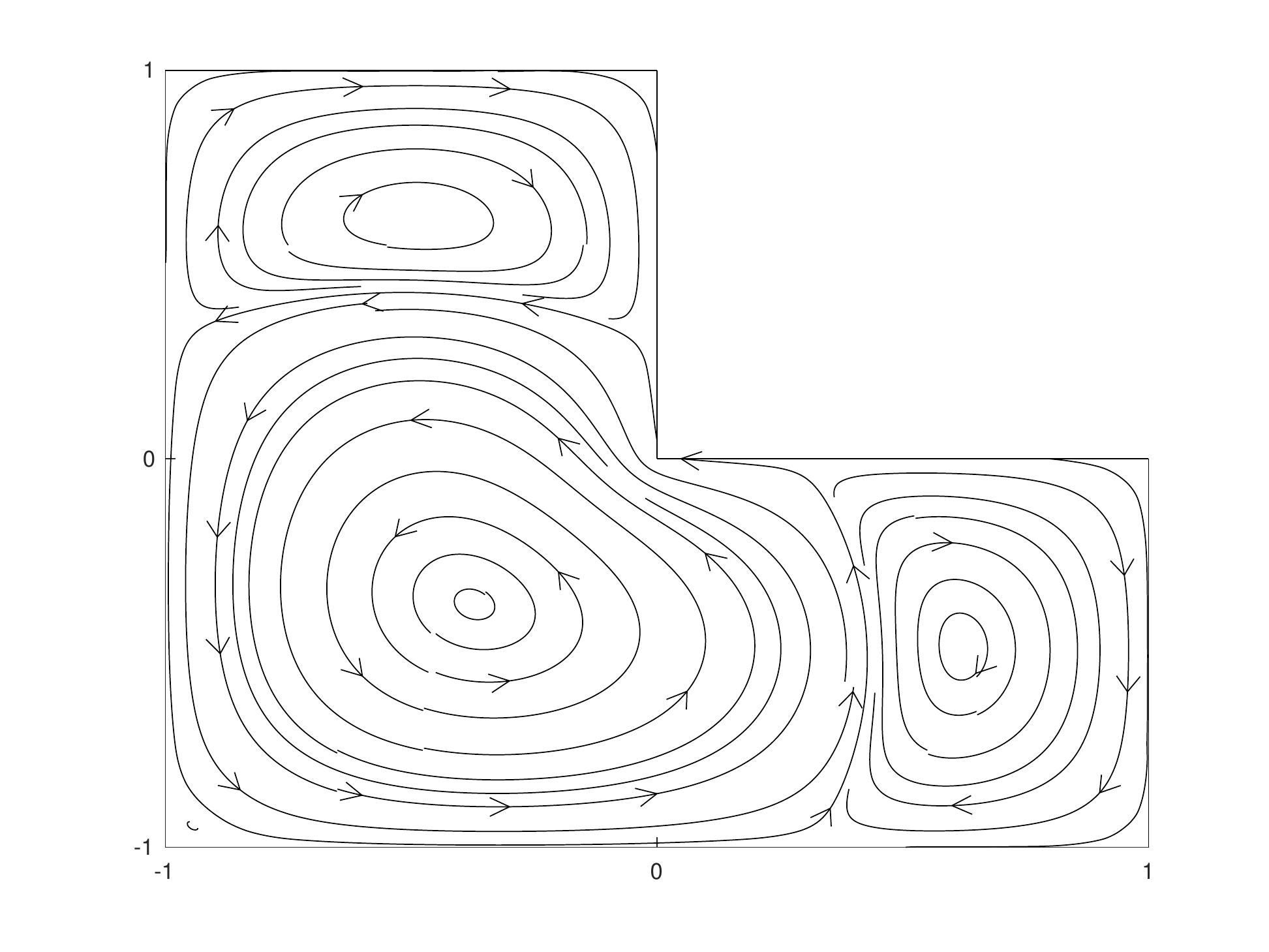}
\label{figLsh22a}
}
\subfigure[]{
\includegraphics[width = \figurewidthb\textwidth]{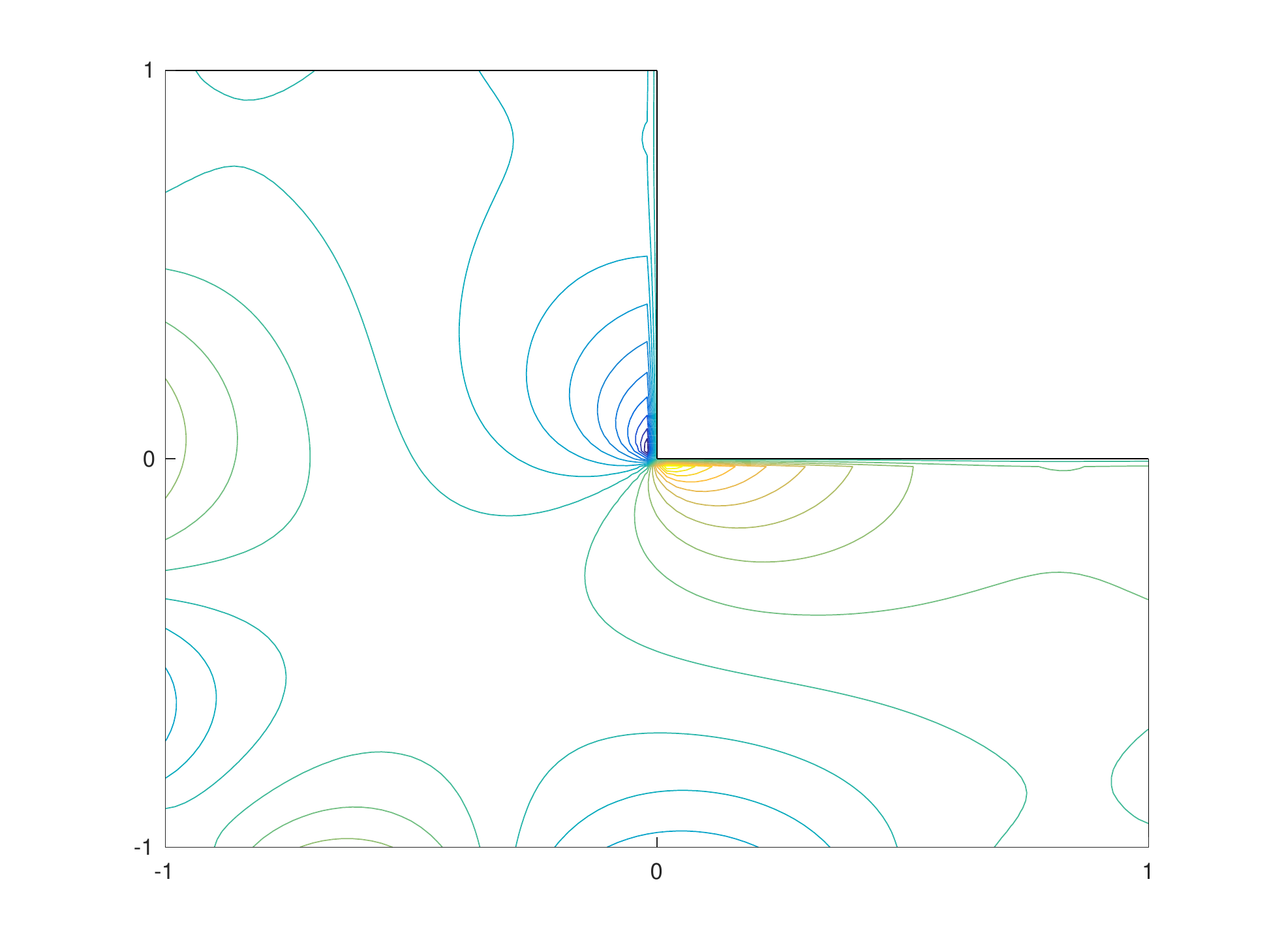}
\label{figLsh22c}
}
\caption[]
{\subref{figLsh22a}Streamline plot of the discrete eigenfunction $\bm{u}_{\ell}$. 
\subref{figLsh22c} Plot of discrete pressure $p_{\ell}=-\mbox{tr}\bm{\sigma}_{\ell}/2$.}
\label{figexLsh212}
\end{figure}
The second example is for the L-shaped domain with
$\Omega=(-1,1)^2\setminus [0,1]^2$.  Here, the 
domain in nonconvex and has a re-entrant corner
at the origin, which causes a singularity in the first eigenfunction.
To compute the eigenvalue error of the first eigenvalue,  
we take $\lambda=32.13269465$ as a reference value,
cf. to the values in \cite{BMS2014,SZ2017} scaled by $1/4$ due to the differently sized domain.
The convergence results for uniform meshes in Figure \ref{figLsh122a}  and \ref{figLsh122c}
show reduced orders of convergence $\mathcal{O}(N^{-0.544})$ 
for all eigenvalue errors 
and both error estimators. 
Furthermore, the convergence results based on 
adaptive refinement recover optimal higher order convergence $\mathcal{O}(N^{-3})$ 
of the post-processed eigenvalues  $\lambda_{\ell}^*$ 
and $\tilde{\lambda}_{\ell}$.
In both cases the corresponding a~posteriori error estimators $\eta_\ell^2$ and $\mu_\ell^2$
are reliable and efficient and close to the true error. 
Moreover, the eigenvalue errors for adaptive refinement are several orders of magnitude
below the eigenvalue errors for uniform refinement, which
illustrates the importance of adaptive mesh refinement.
Figures \ref{figLsh122aw} and  \ref{figLsh122cw} show two adaptively refined 
meshes for the a~posteriori error estimators $\eta_{\ell}^2$ and $\mu_{\ell}^2$, respectively.
Both meshes show strong refinement toward the origin.
Figures \ref{figLsh22a} and \ref{figLsh22c} show the discrete velocity and 
pressure as a streamline plot computed on an adaptive mesh.
%
\subsection{Slit domain}\label{testslitd} 
\begin{figure}[tbp]
\centering
\subfigure[]{
\includegraphics[width = \figurewidtha\textwidth]{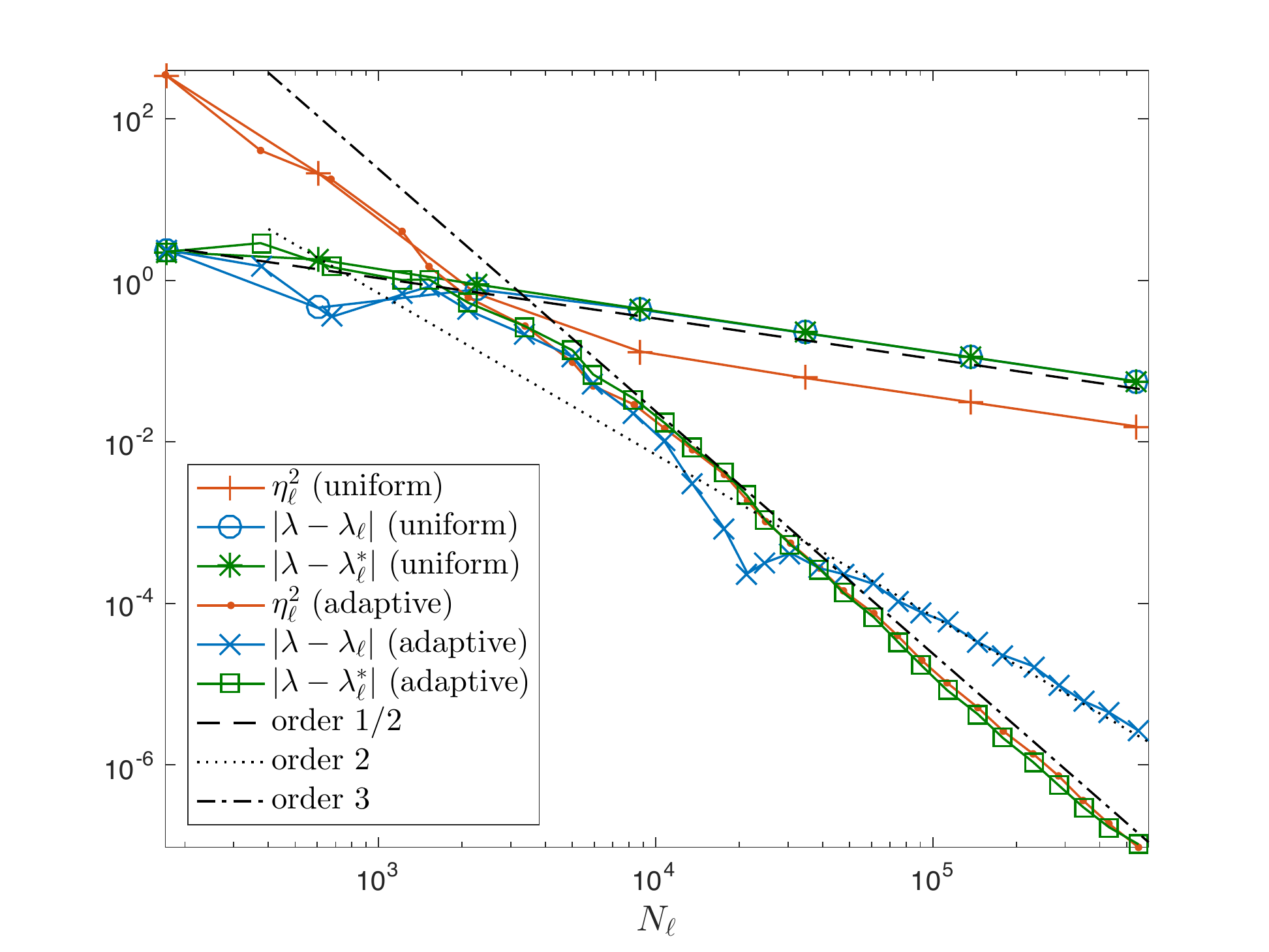}
\label{figSl122a}
}
\subfigure[]{
\includegraphics[width = \figurewidtha\textwidth]{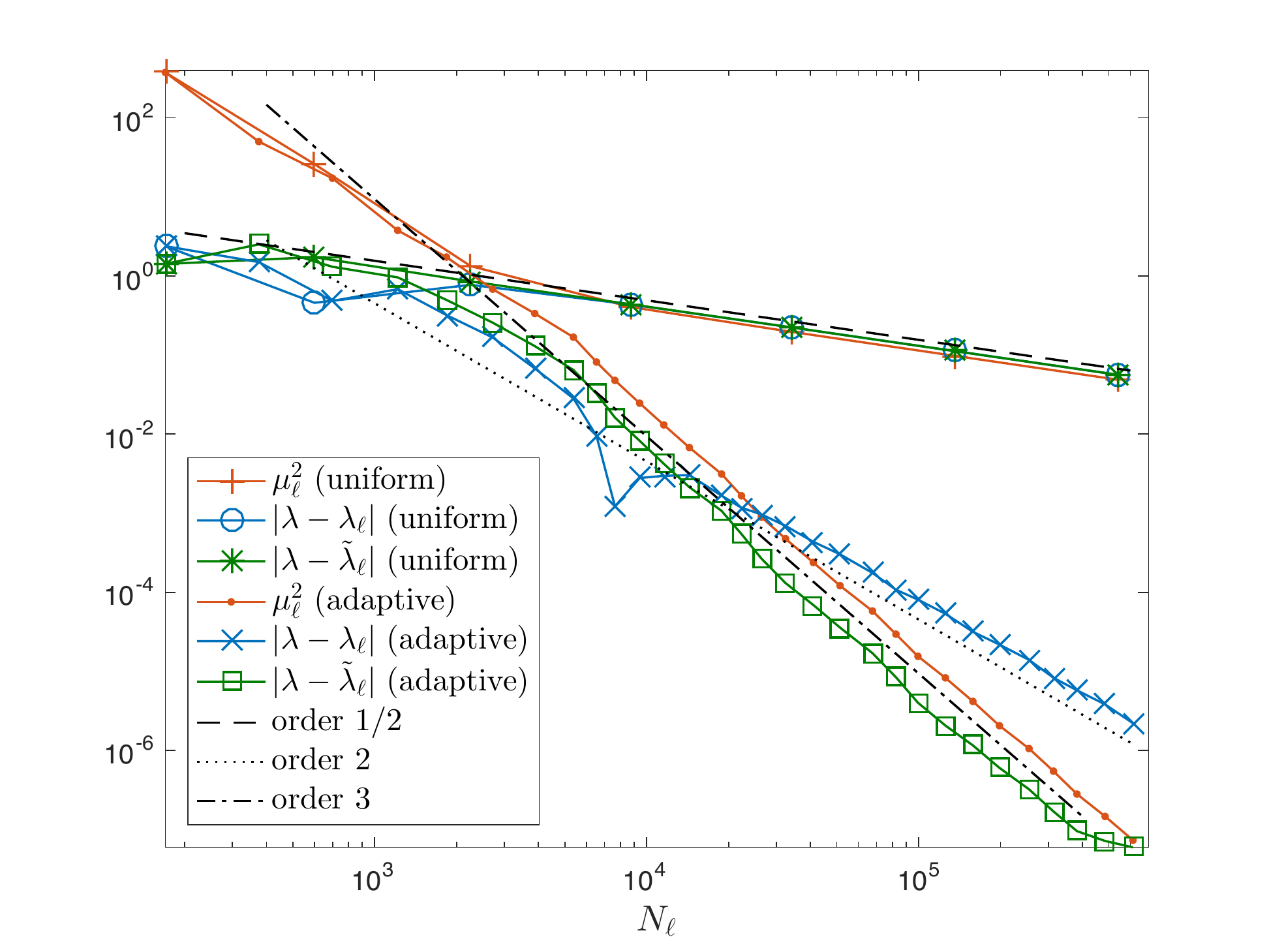}
\label{figSl122c}
}
\caption[]
{ Convergence history of \subref{figSl122a} $|\lambda-\lambda_{\ell}|, |\lambda-\lambda_{\ell}^*|$, $\eta_{\ell}^2$, \subref{figSl122c} $|
\lambda-\lambda_{\ell}|, |\lambda-\tilde{\lambda}_{\ell}|$ and 
$\mu_{\ell}^2$ on uniformly and adaptively refined meshes for the slit 
domain. }
\label{figexSl21}
\end{figure}
\begin{figure}[tbp]
\centering
\subfigure[]{
\includegraphics[width = \figurewidthb\textwidth]{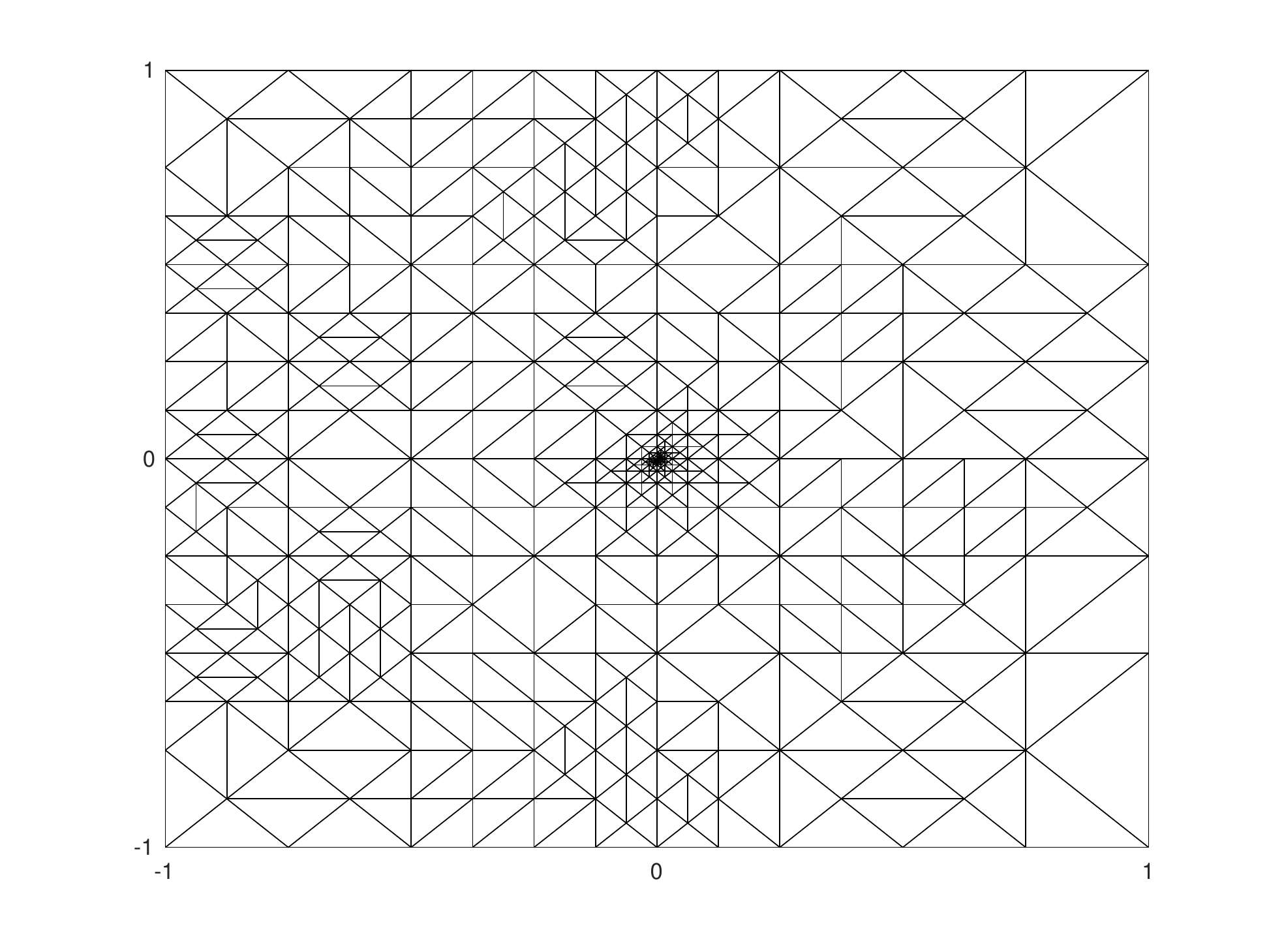}
\label{figSl122aw}
}
\subfigure[]{
\includegraphics[width = \figurewidthb\textwidth]{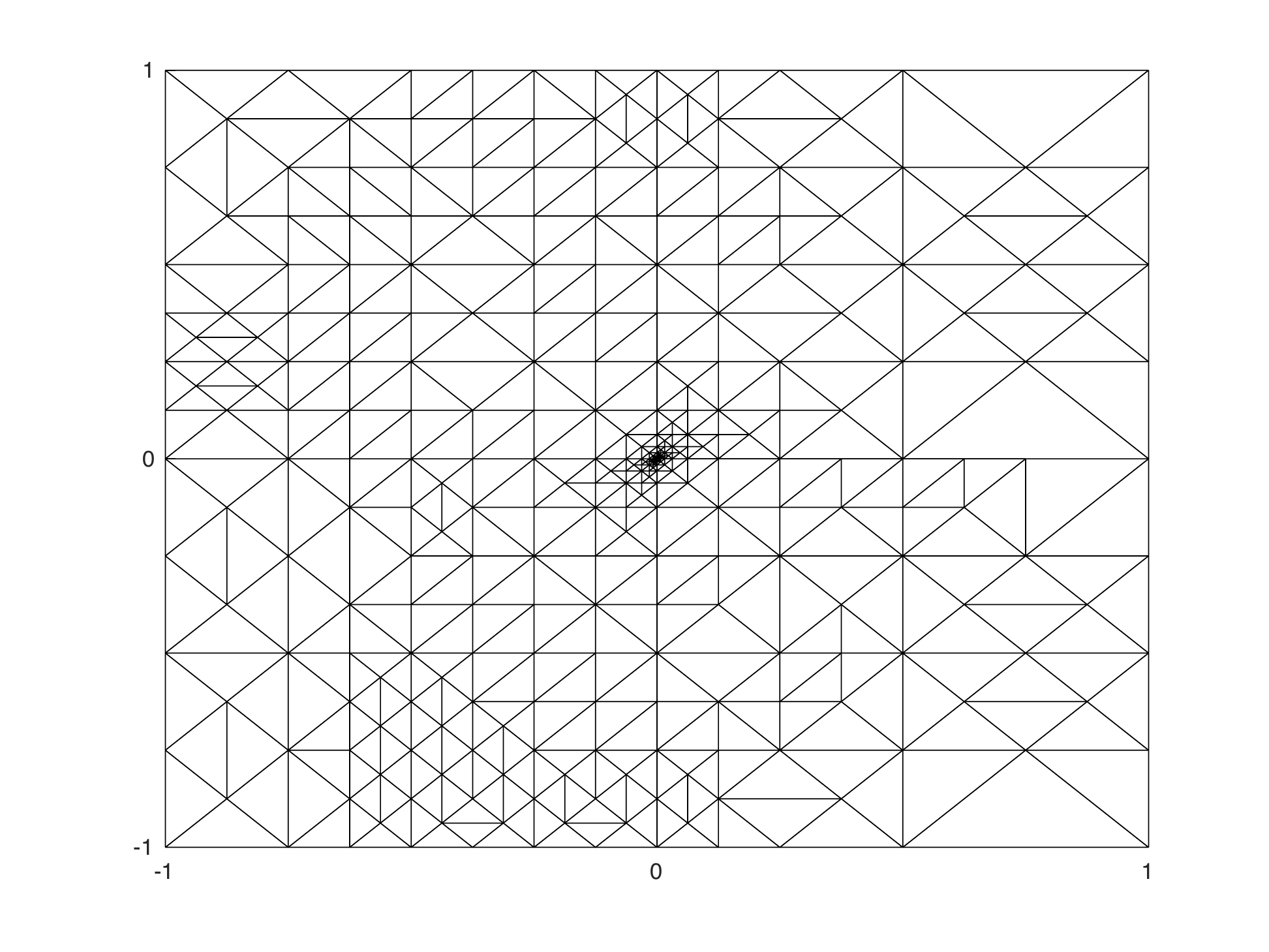}
\label{figSl122cw}
}
\caption[]
{\subref{figSl122aw} Adaptive refined meshes for $\eta_{\ell}^2$ with $375$ nodes, 
\subref{figSl122cw} Adaptive refined meshes for $\mu_{\ell}^2$ with $336$ nodes.}
\label{figexSl21w}
\end{figure}
\begin{figure}[tbp]
\centering
\subfigure[]{
\includegraphics[width = \figurewidthb\textwidth]{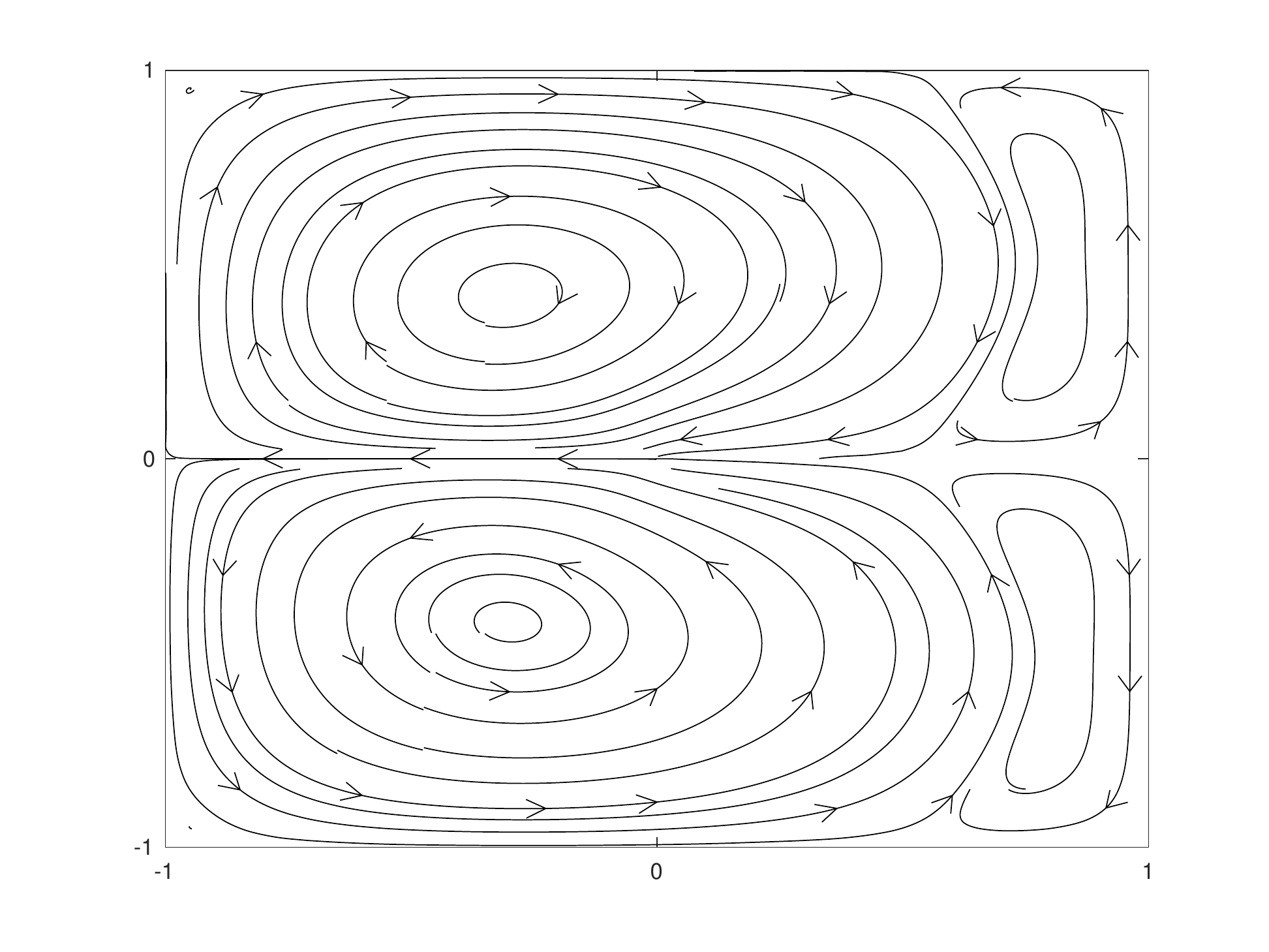}
\label{figSl22a}
}
\subfigure[]{
\includegraphics[width = \figurewidthb\textwidth]{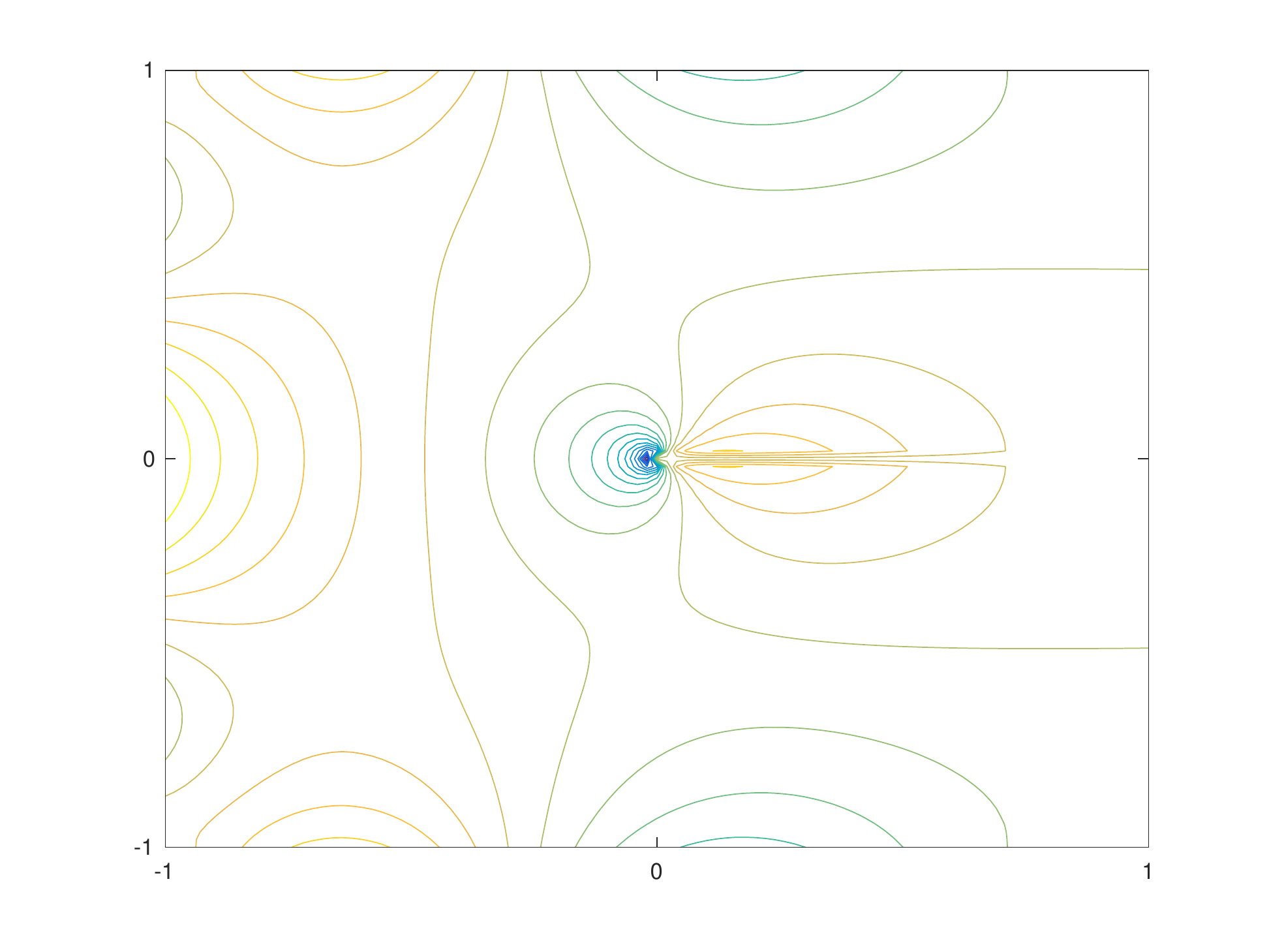}
\label{figSl22c}
}
\caption[]
{\subref{figSl22a}Streamline plot of the discrete eigenfunction $\bm{u}_{\ell}$. 
\subref{figSl22c} Plot of discrete pressure $p_{\ell}=-\mbox{tr}\bm{\sigma}_{\ell}/2$.}
\label{figexSl212}
\end{figure}
Finally,  we consider the slit domain $\Omega=(-1,1)^2\setminus\{0\le x\le 1,y=0\}$,
with maximal re-entrant corner of angle $2\pi$.
We take $\lambda=29.9168629$ as a reference value for the first eigenvalue. 
Again the first eigenfunction is singular.
For uniform meshes in Figure \ref{figSl122a} and \ref{figSl122c}, the convergence results show 
suboptimal convergence $\mathcal{O}(N^{-1/2})$ 
for the eigenvalue errors and error estimators. 
On the contrary, the convergence results, which are based on 
adaptive refinement, achieve optimal convergence $\mathcal{O}(N^{-3})$ 
for the post-processed eigenvalues $\lambda_{\ell}^*$ and $\tilde{\lambda}_{\ell}$, 
and for the a~posteriori error estimators $\eta_{\ell}^2$ and $\mu_{\ell}^2$.
In Figures \ref{figSl122a} and \ref{figSl122c}, the graphs for the 
estimators $\eta_{\ell}^2$ and $\mu_{\ell}^2$ are parallel to the eigenvalue errors 
of the eigenvalues $\lambda_{\ell}^*$ and $\tilde{\lambda}_{\ell}$,
thus it confirms that both error 
estimator are numerically reliable and efficient. 
Note that the efficiency index corresponding to $\eta_{\ell}^2$ 
in Figure \ref{figSl122a} and the efficiency index corresponding $\mu_{\ell}^2$ 
in Figure \ref{figSl122c} are close to one in case of 
adaptive mesh refinement. 
Figures \ref{figSl122aw} and \ref{figSl122cw} display adaptively refined meshes for 
the proposed error estimators $\eta_{\ell}^2$ and $\mu_{\ell}^2$, 
which both show strong refinement toward the re-entrant corner. 
The computed velocity and discrete pressure are shown in Figures \ref{figSl22a} 
and \ref{figSl22c} as a streamline plot on an adaptive mesh. 



\begin{thebibliography}{10}

\bibitem{MAJTO}
M.~Ainsworth and J.~T. Oden.
\newblock {\em A posteriori error estimation in finite element analysis}.
\newblock Pure and Applied Mathematics (New York). John Wiley \& Sons, New
  York, 2000.

\bibitem{AMGVM}
M.~G. Armentano and V.~Moreno.
\newblock A posteriori error estimates of stabilized low-order mixed finite
  elements for the {S}tokes eigenvalue problem.
\newblock {\em J. Comput. Appl. Math.}, 269:132--149, 2014.

\bibitem{ADRW}
D.~N. Arnold and R.~Winther.
\newblock Mixed finite elements for elasticity.
\newblock {\em Numer. Math.}, 92(3):401--419, 2002.

\bibitem{IBJO}
I.~Babu\v{s}ka and J.~Osborn.
\newblock Eigenvalue problems.
\newblock In {\em Handbook of numerical analysis. Volume II: Finite element
  methods (Part 1)}, pages 641--787. Amsterdam etc.: North-Holland, 1991.

\bibitem{BPEBPT}
P.~E. Bj{\o}rstad and B.~P. Tj{\o}stheim.
\newblock High precision solutions of two fourth order eigenvalue problems.
\newblock {\em Computing}, 63(2):97--107, 1999.

\bibitem{DB}
D.~Boffi.
\newblock Finite element approximation of eigenvalue problems.
\newblock {\em Acta Numer.}, 19:1--120, 2010.

\bibitem{BMS2014}
S.~C. Brenner, P.~Monk, and J.~Sun.
\newblock {$C^0$} interior penalty {G}alerkin method for biharmonic eigenvalue
  problems.
\newblock In {\em Spectral and high order methods for partial differential
  equations---{ICOSAHOM} 2014}, volume 106 of {\em Lect. Notes Comput. Sci.
  Eng.}, pages 3--15. Springer, Cham, 2015.

\bibitem{FB}
F.~Brezzi.
\newblock On the existence, uniqueness and approximation of saddle-point
  problems arising from {L}agrangian multipliers.
\newblock {\em Rev. Fran\c{c}aise Automat. Informat. Recherche Op\'erationnelle
  S\'er. Rouge}, 8({\rm R}-2):129--151, 1974.

\bibitem{BF}
F.~Brezzi and M.~Fortin.
\newblock {\em Mixed and hybrid finite element methods}, volume~15 of {\em
  Springer Series in Computational Mathematics}.
\newblock Springer-Verlag, New York, 1991.

\bibitem{CC}
C.~Carstensen.
\newblock A unifying theory of a posteriori finite element error control.
\newblock {\em Numer. Math.}, 100(4):617--637, 2005.

\bibitem{CCJGEJP}
C.~Carstensen, J.~Gedicke, and E.-J. Park.
\newblock Numerical experiments for the {A}rnold-{W}inther mixed finite
  elements for the {S}tokes problem.
\newblock {\em SIAM J. Sci. Comput.}, 34(4):A2267--A2287, 2012.

\bibitem{CWQL}
W.~Chen and Q.~Lin.
\newblock Approximation of an eigenvalue problem associated with the {S}tokes
  problem by the stream function-vorticity-pressure method.
\newblock {\em Appl. Math.}, 51(1):73--88, 2006.

\bibitem{Doerfler}
W.~D\"orfler.
\newblock A convergent adaptive algorithm for {P}oisson's equation.
\newblock {\em SIAM J. Numer. Anal.}, 33(3):1106--1124, 1996.

\bibitem{HJZZYY}
J.~Han, Z.~Zhang, and Y.~Yang.
\newblock A new adaptive mixed finite element method based on residual type a
  posterior error estimates for the {S}tokes eigenvalue problem.
\newblock {\em Numer. Methods Partial Differential Equations}, 31(1):31--53,
  2015.

\bibitem{PH}
P.~Huang.
\newblock Lower and upper bounds of {S}tokes eigenvalue problem based on
  stabilized finite element methods.
\newblock {\em Calcolo}, 52(1):109--121, 2015.

\bibitem{PHYHXF}
P.~Huang, Y.~He, and X.~Feng.
\newblock Numerical investigations on several stabilized finite element methods
  for the {S}tokes eigenvalue problem.
\newblock {\em Math. Probl. Eng.}, pages Art. ID 745908, 14, 2011.

\bibitem{KP}
O.~A. Karakashian and F.~Pascal.
\newblock A posteriori error estimates for a discontinuous {G}alerkin
  approximation of second-order elliptic problems.
\newblock {\em SIAM J. Numer. Anal.}, 41(6):2374--2399, 2003.

\bibitem{ARPACK}
R.~Lehoucq, D.~Sorensen, and C.~Yang.
\newblock {\em {A}{R}{P}{A}{C}{K} Users' Guide: Solution of Large-Scale
  Eigenvalue Problems with Implicitly Restarted Arnoldi Methods}.
\newblock Society for Industrial and Applied Mathematics, Philadelphia, PA.
  USA, 1998.

\bibitem{LHGWWSYN}
H.~Liu, W.~Gong, S.~Wang, and N.~Yan.
\newblock Superconvergence and a posteriori error estimates for the {S}tokes
  eigenvalue problems.
\newblock {\em BIT}, 53(3):665--687, 2013.

\bibitem{LCMLRS}
C.~Lovadina, M.~Lyly, and R.~Stenberg.
\newblock A posteriori estimates for the {S}tokes eigenvalue problem.
\newblock {\em Numer. Methods Partial Differential Equations}, 25(1):244--257,
  2009.

\bibitem{MSDMRR}
S.~Meddahi, D.~Mora, and R.~Rodr{\'\i}guez.
\newblock A finite element analysis of a pseudostress formulation for the
  {S}tokes eigenvalue problem.
\newblock {\em IMA J. Numer. Anal.}, 35(2):749--766, 2015.

\bibitem{BJJP}
B.~Mercier, J.~Osborn, J.~Rappaz, and P.-A. Raviart.
\newblock Eigenvalue approximation by mixed and hybrid methods.
\newblock {\em Math. Comp.}, 36(154):427--453, 1981.

\bibitem{SZ}
L.~R. Scott and S.~Zhang.
\newblock Finite element interpolation of nonsmooth functions satisfying
  boundary conditions.
\newblock {\em Math. Comp.}, 54(190):483--493, 1990.

\bibitem{SZ2017}
J.~Sun and A.~Zhou.
\newblock {\em Finite element methods for eigenvalue problems}.
\newblock Monographs and Research Notes in Mathematics. CRC Press, Boca Raton,
  FL, 2017.

\bibitem{OTDBRC}
O.~T\"urk, D.~Boffi, and R.~Codina.
\newblock A stabilized finite element method for the two-field and three-field
  {S}tokes eigenvalue problems.
\newblock {\em Comput. Methods Appl. Mech. Engrg.}, 310:886--905, 2016.

\bibitem{RV}
R.~Verf\"urth.
\newblock {\em {A review of a posteriori error estimation and adaptive
  mesh-refinement techniques.}}
\newblock Chichester: John Wiley \& Sons; Stuttgart: B. G. Teubner, 1996.

\end{thebibliography}
\end{document}